\documentclass[preprint,12pt,authoryear,nopreprintline]{elsarticle}

\usepackage{amsthm,amsmath,amsfonts,amssymb,amsbsy,bbm,stackengine,tikz,graphicx,color,dsfont}
\usepackage{booktabs,multirow} 
\usepackage{dsfont}

\newcommand{\y}{\mathbf{y}}

\newcommand{\x}{\mathbf{x}}

\def \A {\mathbf{A}}

\def \B {\mathbf{B}}

\def \C {\mathbf{C}}

\def \I {\mathbf{I}}

\def \R {\mathbf{R}}

\def \S {\mathbf{S}}

\def \w {\mathbf{w}}
\def \x {\mathbf{x}}

\def \y {\mathbf{y}}

\def \z {\mathbf{z}}

\def \Ecal {\mathcal{E}}
\def \Fcal {\mathcal{F}}
\def \Gcal {\mathcal{G}}
\def \Hcal {\mathcal{H}}
\def \Ical {\mathcal{I}}

\def \Ncal {\mathcal{N}}
\def \Ocal {\mathcal{O}}

\def \Cbb {\mathbb{C}}
\def \Ebb {\mathbb{E}}

\def \Nbb {\mathbb{N}}
\def \Pbb {\mathbb{P}}
\def \Rbb {\mathbb{R}}

\def \Zbb {\mathbb{Z}}

\def \erm {\mathrm{e}}
\def \irm {\mathrm{i}}

\def \epsilonbs {\boldsymbol{\epsilon}}

\def \xibs {\boldsymbol{\xi}}

\def \Pibs {\boldsymbol{\Pi}}

\def \Sigmabs {\boldsymbol{\Sigma}}

\def \Tr {\mathrm{tr}\ }
\def \Prob {\mathbb{P}}

\def \diag {\mathrm{dg}}

\renewcommand{\Im}{\mathrm{Im}}
\renewcommand{\Re}{\mathrm{Re}}

\newcommand*\diff{\mathop{}\!\mathrm{d}}

\newtheorem{theorem}{Theorem}
\newtheorem{proposition}{Proposition}
\newtheorem{assumption}{Assumption}
\newtheorem{lemma}[theorem]{Lemma}

\newdefinition{remark}{Remark}

\begin{document}

\begin{frontmatter}



\title{On the asymptotic distribution of the maximum sample spectral coherence of Gaussian time series in the high dimensional regime\tnoteref{t1}}
\tnotetext[t1]{This work is funded by ANR Project HIDITSA, reference ANR-17-CE40-0003.}


\author[1]{Philippe Loubaton}
\ead{philippe.loubaton@univ-eiffel.fr}

\author[2]{Alexis Rosuel\corref{cor1}}
\ead{alexis.rosuel@univ-eiffel.fr}

\author[3]{Pascal Vallet}
\ead{pascal.vallet@bordeaux-inp.fr}

\address[1]{Laboratoire d'Informatique Gaspard Monge (CNRS, Univ. Gustave-Eiffel), 5 Bd. Descartes 77454 Marne-la-Vallée (France)}
\address[2]{Laboratoire d'Informatique Gaspard Monge (CNRS, Univ. Gustave-Eiffel), 5 Bd. Descartes 77454 Marne-la-Vallée (France)}
\address[3]{Laboratoire de l'Intégration du Matériau au Système (CNRS, Univ. Bordeaux, Bordeaux INP), 351, Cours de la Libération 33405 Talence (France)}

\begin{abstract}
We investigate the asymptotic distribution of the maximum of a frequency smoothed estimate of the spectral coherence of a $M$-variate complex Gaussian time series with mutually independent components when the dimension $M$ and the number of samples $N$ both converge to infinity. If $B$ denotes the smoothing span of the underlying smoothed periodogram estimator, a type I extreme value limiting distribution is obtained under the rate assumptions $\frac{M}{N} \rightarrow 0$ and $\frac{M}{B}\to c \in (0,+\infty)$. This result is then exploited to build a statistic with controlled asymptotic level for testing independence between the $M$ components of the observed time series. Numerical simulations support our results. 
\end{abstract}




\end{frontmatter}


\section{Introduction}
\subsection{The addressed problem and the results}

\sloppy
                   
We consider $M$ jointly stationary complex Gaussian time series $\left(y_{1,n}\right)_{n \in \Zbb},\ldots,\left(y_{M,n}\right)_{n \in \Zbb}$ and for all $i,j \in \{1,\ldots,M\}$, we denote by $s_{ij}$ and $c_{ij}$ the spectral density and spectral coherence between $\left(y_{i,n}\right)_{n \in \Zbb}$ and $\left(y_{j,n}\right)_{n \in \Zbb}$ given respectively by
  \begin{align}
    s_{ij}(\nu) = \sum_{u \in \Zbb} r_{ij}(u) \erm^{- \irm 2 \pi u \nu}
    \notag
  \end{align}
  and
\begin{align}
    c_{ij}(\nu) = \frac{s_{ij}(\nu)}{\sqrt{s_{ii}(\nu) s_{jj}(\nu)}}
    \notag
\end{align}
for all $\nu \in [0,1]$, where $r_{ij}(u) = \Ebb[y_{i,n+u} \overline{y}_{j,n}]$.
Assuming $N$ observations $\left(y_{1,n}\right)_{n = 1,\ldots,N},\ldots,\left(y_{M,n}\right)_{n =1,\ldots,N}$ are available for each time series, we consider the frequency smoothed estimate $\hat{s}_{ij}$ of $s_{ij}$ given by
\begin{align}
  \hat{s}_{ij}(\nu) = \frac{1}{B+1} \sum_{b=-B/2}^{B/2} \xi_{y_i}\left(\nu + \frac{b}{N}\right)  \overline{\xi_{y_j}\left(\nu + \frac{b}{N}\right)},
  \label{eq:sijhat}
\end{align}
where $B$ is an even integer representing the smoothing span, and where
\begin{align}
  \xi_{y_i}(\nu) = \frac{1}{\sqrt{N}} \sum_{n=1}^{N} y_{i,n} \erm^{-2 \irm \pi(n-1)\nu}
  \notag
\end{align}
denotes the normalized Fourier transform of $\left(y_{i,n}\right)_{n = 1,\ldots,N}$.
The corresponding sample estimate of the spectral coherence is defined as
\begin{align}
  \hat{c}_{ij}(\nu) = \frac{\hat{s}_{ij}(\nu)}{\sqrt{\hat{s}_{ii}(\nu) \hat{s}_{jj}(\nu)}}.
  \notag
\end{align}
Under the hypothesis 
\[
    \Hcal_0: (y_{1,n})_{n \in \mathbb{Z}}, \ldots,  (y_{M,n})_{n \in \mathbb{Z}} \text{ are mutually uncorrelated},
\]
we evaluate the behaviour of the Maximum Sample Spectral Coherence (MSSC) defined by
\[
    \max_{1\le i<j\le M}\max_{\nu \in \Gcal} |\hat{c}_{ij}(\nu)|
\]
where 
\[
    \Gcal := \left\{k\frac{B+1}{N}: k \in \Nbb, 0\le k\le \frac{N}{B+1}\right\}
\]
is the subset of the Fourier frequencies
\begin{align*}
\Fcal:=\left\{\frac{k}{N}: k \in \Nbb, 0\le k\le N-1\right\}
\end{align*}
with elements spaced by a distance $(B+1)/N$.
Our study is conducted in the asymptotic regime where $M = M(N)$ and $B = B(N)$ are both functions of $N$ such that for some $\rho \in (0,1)$, $M \asymp N^{\rho}$ and $B \asymp N^{\rho}$
as $N\to\infty$
\footnote{For two sequences $(x_n)_{n \geq 1}, (y_n)_{n \geq 1}$, we denote by $x_n \asymp y_n$ if there exists $k_1,k_2 > 0$ such that $k_1 |y_n| \leq |x_n| \leq k_2 |y_n|$ for all large $n$.}
, while the ratio $M/B$ converges to some constant $c\in(0,+\infty)$. It is established that, under $\Hcal_0$ and proper assumptions on the time series
$(y_{1,n})_{n \in \Zbb},\ldots,(y_{M,n})_{n \in \Zbb}$, for any $t\in\Rbb$:
\begin{align}
  \Prob\left((B+1)\max_{(i,j,\nu)\in\Ical}|\hat{c}_{ij}(\nu)|^2 \le t + \log \frac{N}{B+1} + \log \frac{M(M-1)}{2} \right) \xrightarrow[N\to+\infty]{} e^{-e^{-t}}
  \label{eq:result-paper}
\end{align}   
where
\begin{equation}
\label{definition:Ical_N}
    \Ical := \{(i,j,\nu) : i,j\in[M] \text{ such that }i<j,\ \nu\in\Gcal\}
\end{equation}
with $[M]=\{1,\ldots,M\}$.

In other words, under proper normalization and centering, $\max_{(i,j,\nu)\in\Ical}|\hat{c}_{ij}(\nu)|^2$ follows asymptotically a Gumbel distribution (see \cite{embrechts2013modelling} or \cite{resnick2013extreme} for a general theory of extreme value distributions).

\subsection{Motivation}
This paper is motivated by the problem of testing the independence of a large number of Gaussian time series. Since hypothesis $\Hcal_0$ can be equivalently formulated as

\begin{align}
  \Hcal_0 : \max_{1\le i<j\le M} \max_{\nu\in[0,1]} |s_{ij}(\nu)|^2 = 0,
  \notag
\end{align}
or  by
\begin{align}
  \Hcal_0 : \max_{1\le i<j\le M} \max_{\nu\in[0,1]} |c_{ij}(\nu)|^2 = 0,
  \notag
\end{align}
this suggests to compute consistent estimators of these quantities, and test their closeness to zero.

Our choice of the high-dimensional regime defined above is motivated as follows. Under mild assumptions on the memory of the time series $((y_{m,n})_{n\in\Zbb})_{m\geq 1}$, in the low-dimensional regime where $N\to+\infty$ and $M$ is fixed, it can be shown that the sample spectral coherence matrix
\begin{align}
  \hat{\C}(\nu) = \left(\hat{c}_{i,j}(\nu)\right)_{i,j=1,\ldots,M}
  \label{equation:definition_coherency}
\end{align}
is a consistent estimate (in spectral norm for instance) of the spectral coherence matrix
\begin{align}
  \C(\nu) = \left(c_{i,j}(\nu)\right)_{i,j=1,\ldots,M}
  \notag
\end{align}
as long as $B\to+\infty$ and $B/N \to 0$ (up to some additional logarithmic terms). In practice, this asymptotic regime and the underlying predictions are relevant as long as the ratio $M/N$ is small enough. If this condition is not met, test
statistics based on $\hat{\C}(\nu)$ may be of delicate use, as the choice of the smoothing span $B$ must meet the constraints $B \gg M$ (because $B$ is supposed to converge towards $+\infty$) as well as $B \ll N$ (because $B/N$ is supposed to converge towards $0$).
Nowadays, for many practical applications involving high dimensional signals and/or a moderate sample size, the ratio $M/N$ may not be small enough to be able to choose $B$ so as to meet $B \gg M$ and $B \ll N$. In this situation, one may rely on the more relevant high dimensional regime in which $M,B,N$ converge to infinity such that $M/B$ converges to a positive constant while $B/N$ converges to zero.

\subsection{On the literature}

Correlation tests using spectral approaches have been studied in several papers, see e.g. \cite{wahba1971some}, \cite{eichler2008testing} and the references therein.

More recently, an approach similar to the one of this paper has been explored in \cite{wu2018asymptotic}, where the maximum of the sample spectral coherence, when using lag-window estimates of the spectral density, is studied. In the low-dimensional regime where $M$ is fixed and $N\to\infty$, it is proved that the distribution of such statistic under $\Hcal_0$, after proper centering and normalization, converges to the Gumbel distribution.
We also mention other related papers exploring the asymptotic behaviour of various spectral density estimates in the low-dimensional regime: \cite{woodroofe1967maximum}, \cite{rudzkis1985distribution}, \cite{shao2007asymptotic}, \cite{lin2009maxima} and \cite{liu2010asymptotics}.

In the high-dimensional regime when $M$ is a function of $N$ such that $M:=M(N)\to+\infty$, few results on the behaviour of correlation test statistics in the spectral domain are known.
\cite{loubaton2021large} proved that under $\Hcal_0$ and mild assumptions on the underlying time series, the empirical eigenvalue distribution of $\hat{\C}(\nu)$ defined in \eqref{equation:definition_coherency} converges weakly almost surely towards the Marcenko-Pastur distribution, which can be exploited to build test statistics based on linear spectral statistics of $\hat{\C}(\nu)$.
In \cite{rosuel2020frequency}, a consistent test statistic based on the largest eigenvalue of $\hat{\C}(\nu)$ was derived for the problem of detecting the presence of a signal with low rank spectral density matrix within a noise with uncorrelated components. 

In the asymptotic regime where $\frac{M}{N} \rightarrow \gamma$, \cite{pan2014testing} proposed to test hypothesis $\mathcal{H}_0$ when the components of $\y$ share the same spectral density. In this case, the rows of the $M \times N$ matrix $(\y_1, \ldots, \y_N)$ are independent and identically distributed under $\mathcal{H}_0$. \cite{pan2014testing} established a central limit theorem for linear spectral statistics of the empirical covariance matrix, and deduced from this a test statistics to check whether $\mathcal{H}_0$ holds or not. We notice that the results of \cite{pan2014testing} are valid in the non Gaussian case.

More results are available in the case where the time series $\left(y_{m,n}\right)_{n \in \Zbb}$, $m\in[M]$, are temporally white. To test the correlation of the $M$ components, one can similarly consider sample estimates of the correlation matrix, and test whether it is close to the identity matrix. Under the asymptotic regime where $\frac{M}{N}\to\gamma\in(0,+\infty)$, \cite{jiang2004asymptotic} showed that the maximum off-diagonal entry of the sample correlation matrix after proper normalization is also asymptotically distributed as Gumbel. The techniques used here for proving \eqref{eq:result-paper} are partly based on this paper.
Other works such as \cite{mestre2017correlation} studied the asymptotic distribution of linear spectral statistics of the correlation matrix, \cite{dette2020likelihood} focused on the behaviour of the determinant of the correlation matrix, and \cite{cai2013optimal} considered a U-statistic and obtained minimax results over some class of alternatives. Some other papers also explored various classes of alternative $\Hcal_1$, among which \cite{fan2019largest}, who showed a phase transition phenomena in the behaviour of the largest off-diagonal entry of the correlation matrix driven by the magnitude of the dependence parameter defined in the alternative class $\Hcal_1$. Lastly, \cite{morales2018asymptotics} studied asymptotic first and second order behaviour of the largest eigenvalues and associated eigenvectors of the sample correlation matrix under a specific alternative spiked model.

\section{Main results}

\subsection{Assumptions}
\label{section:assumptions}

In all the paper we rely on the following assumptions. 

\begin{assumption}[Time series]
\label{assumption:gaussian_y_n}
The time series $(y_{m,n})_{n\in\Zbb}$, $m \geq 1$, are mutually independent, stationary and zero-mean complex Gaussian distributed
\footnote
{
  A complex random variable $Z$ is zero-mean complex Gaussian distributed with variance $\sigma^2$, denoted as $Z \sim \Ncal_{\Cbb}(0,\sigma^2)$, if $\Re(Z)$ and $\Im(Z)$ are i.i.d. $\Ncal(0,\frac{\sigma^2}{2})$
  random variables.
}.

\end{assumption}


For each $m \geq 1$, we denote by $r_m = (r_{m}(u))_{u \in \Zbb}$  (instead of $r_{m,m}$) the covariance sequence of $(y_{m,n})_{n\in\Zbb}$, i.e. $r_{m}(u) = \Ebb[y_{m,n+u} \overline{y_{m,n}}]$, and we formulate the following
assumption on $(r_m)_{m \geq 1}$:
\begin{assumption}[Memory]
\label{assumption:short_memory}
The covariance sequences $(r_m)_{m \geq 1}$ satisfy the uniform short memory condition
\begin{equation*}
  \sup_{m\ge 1}\sum_{u\in\Zbb}(1+|u|)|r_{m}(u)|<+\infty.
\end{equation*}
\end{assumption}
We denote by $s_m(\nu) = \sum_{u \in \Zbb} r_m(u) \erm^{-\irm 2 \pi \nu}$ (instead of $s_{m,m}(\nu)$) the spectral density of $(y_{m,n})_{n \in \Zbb}$ at frequency $\nu\in[0,1]$.
Assumption \ref{assumption:short_memory} of course implies that the function $s_m$ is continously differentiable and that 
\begin{equation}
\label{eq:derivative-spectral-densities}
    \sup_{m \geq 1} \max_{\nu \in [0, 1]} s_m(\nu) < +\infty, \quad \sup_{m \geq 1} \max_{\nu \in [0, 1]} \left|\frac{\diff s_m}{\diff \nu}(\nu)\right| < +\infty.
\end{equation}
Eventually, as the sample spectral coherence of $(y_{i,n})_{n\in\Zbb}$ and $(y_{j,n})_{n\in\Zbb}$ involves a renormalization by the inverse of the estimates of the spectral densities $s_i$ and $s_j$, we also need that $s_i, s_j$ do not vanish, which is the purpose of the following assumption.
\begin{assumption}[Non-vanishing spectrum]
\label{assumption:non_vanishing_spectrum}
The spectral densities are uniformly bounded away from zero, that is
\begin{equation}
\label{eq:s-bounded-away-from-zero}
    \inf_{m\ge1}\min_{\nu\in[0, 1]} s_m(\nu) >0.
\end{equation}
\end{assumption}
By Assumptions \ref{assumption:short_memory} and \ref{assumption:non_vanishing_spectrum}, there exist quantities $s_{\min}$ and $s_{\max}$ such that
\begin{equation}
\label{equation:definition_smin_smax}
    0 < s_{\min} \le \inf_{m\ge1}\min_{\nu\in[0,1]} s_m(\nu) \le \sup_{m\ge1}\max_{\nu\in[0,1]} s_m(\nu) \le s_{\max} < +\infty.
\end{equation}
We now formulate the following assumptions on the growth rate of the quantities $N,M,B$, which describe the high-dimensional regime considered in this paper.
\begin{assumption}[Asymptotic regime]
\label{assumption:rate_NBM}
$B$ and $M$ are functions of $N$ such that there exist positive constants $C_1,C_2\in(0,+\infty)$ and $\rho\in(0,1)$ such that:
\[
    C_1N^\rho \le B,M \le C_2N^\rho
\]
and 
\[
    \frac{M}{B} := c_N \xrightarrow[N\to+\infty]{} c \in (0,+\infty).
\]
\end{assumption}
    
\paragraph{Notations}
Even if the subscript $\cdot_N$ is not always specified, almost all quantities should be remembered to be dependent on $N$. Moreover, $C$ represents a universal constant (i.e. a positive quantity independent of $N,M,B$), whose precise value is irrelevant and which may change from one line to another.

\subsection{Statement of the result}

The main result of this paper, whose proof is deferred to Section \ref{section:proof_main_theorem}, is given in the following theorem.
\begin{theorem}
\label{theorem:main}
Under Assumptions \ref{assumption:gaussian_y_n} -- \ref{assumption:non_vanishing_spectrum}, for any $t\in\Rbb$:
\[
    \Prob\left((B+1)\max_{(i,j,\nu)\in\Ical}|\hat{c}_{ij}(\nu)|^2 \le t + \log \frac{N}{B+1} + \log \frac{M(M-1)}{2} \right) \xrightarrow[N\to+\infty]{} e^{-e^{-t}}.
\]  
\end{theorem}
Thus, Theorem \ref{theorem:main} states that $\max_{(i,j,\nu)\in\Ical}|\hat{c}_{ij}(\nu)|^2$, atfer proper normalization and centering, converges in distribution to a type I extreme value distribution, also known as Gumbel distribution. As it will be clear in the proof, the term $\log \frac{M(M-1)}{2}$ is related to the maximum over $(i,j)$ while the term $\log \frac{N}{B+1}$ is related to the maximum over $\nu\in\Gcal$.

\bigbreak
We now illustrate numerically the above asymptotic result. Consider $M$ independent AR(1) processes, driven by a standard Gaussian white noise, i.e.
\[
    \y_n := \begin{pmatrix}
        y_{1,n} \\
        \vdots \\
        y_{M,n} 
    \end{pmatrix} = \theta \begin{pmatrix}
        y_{1,n-1} \\
        \vdots \\
        y_{M,n-1} 
    \end{pmatrix} + \begin{pmatrix}
        \epsilon_{1,n} \\
        \vdots \\
        \epsilon_{M,n} 
    \end{pmatrix}, \quad \epsilon_{m,n} \overset{i.i.d.}{\sim} \Ncal_\Cbb(0, 1)
\]
with $\theta=0.6$, and $(N,M)=(20000, 500)$. The smoothed periodogram estimators are computed using $B=1000$. We independently draw 10000 samples of the time series $(\y_n)_{n\in[N]}$ and compute the associated MSSC $\max_{(i,j,\nu)\in\Ical_N}|\hat{c}_{ij}(\nu)|^2$. On Figure \ref{figure:mcc_vs_gumbel} are represented the sample cumulative distribution function (cdf) and the histogram of the MSSC against the Gumbel cdf and probability density function (pdf). We indeed observe that the rescaled distribution of $\max_{(i,j,\nu)\in\Ical_N}|\hat{c}_{ij}(\nu)|^2$ is close to the Gumbel distribution.

\begin{figure}[!ht]
\centering
     \includegraphics[width=1.0\textwidth]{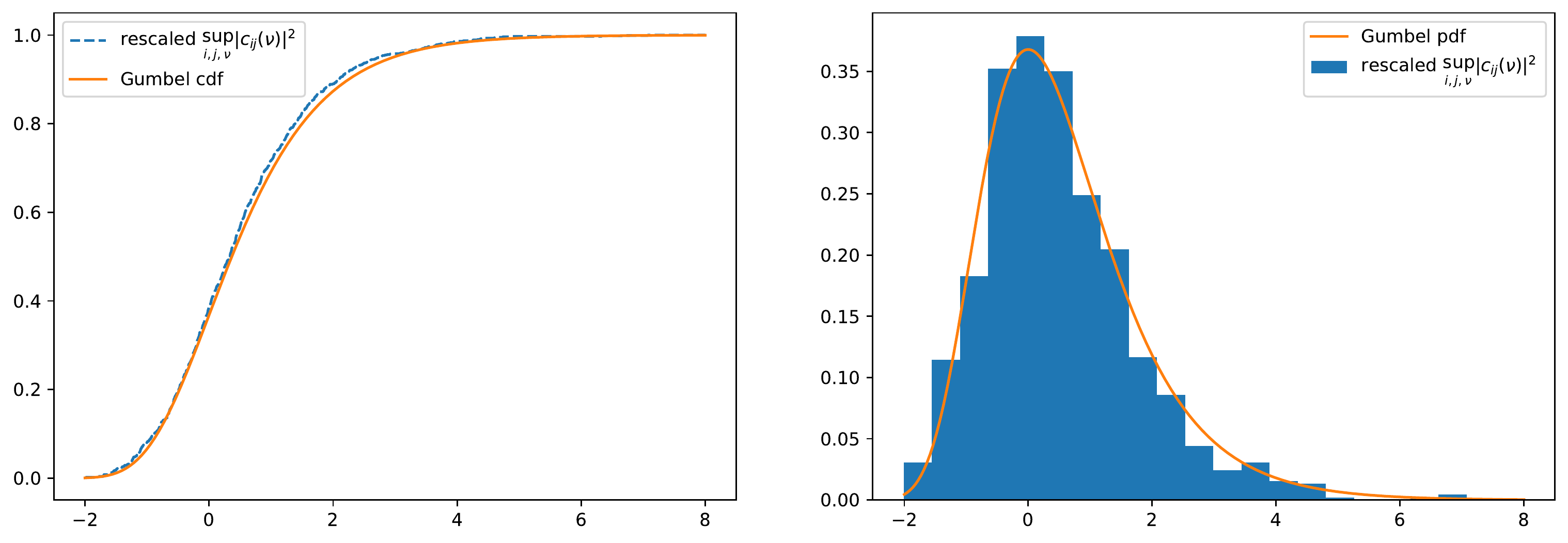}
      \caption{sample cdf and histogram of the MSSC as defined in Theorem \ref{theorem:main} vs Gumbel distribution. }
      \label{figure:mcc_vs_gumbel}. 
\end{figure}

\section{Application to testing}
\subsection{New proposed test statistic}

Theorem \ref{theorem:main} can be used to design a new independence test statistic with controlled asymptotic level in the proposed high-dimensional regime.

Define $q_\alpha$ the $\alpha$--quantile of the Gumbel distribution: $q_\alpha = F^{-1}(\alpha)$ where 
\[
    F(x) = \exp(-\exp(-x)).
\]
The test statistic $T_N^{(\mathrm{MSSC})}$ defined by
\begin{equation}
\label{definition:T_MCC}
    T_N^{(\mathrm{MSSC})} = \mathds{1}\left(\max_{(i,j,\nu)\in\Ical}|\hat{c}_{ij}(\nu)|^2>\frac{q_{1-\alpha}+\log\frac{N}{B+1}+\log\frac{M(M-1)}{2}}{B+1}\right)
\end{equation}
satisfies, as a direct consequence of Theorem \ref{theorem:main}, $\lim_{N\to+\infty}\Prob[T_N^{(\mathrm{MSSC})}=1]=\alpha$ under $\Hcal_0$. 

\subsection{Type I error}

In order to test the independence of the signals $((y_{m,n})_{n\in\Zbb})_{m=1,\ldots,M}$, we consider the statistic $T_N^{(\mathrm{MSSC})}$ defined in \eqref{definition:T_MCC}. On Table \ref{table:type_I_error} are presented the sample type I errors of $T_N^{(\mathrm{MSSC})}$ with different combinations of sample sizes and dimensions ($\rho=0.7$ and $\frac{M}{B+1}=0.5$), when the nominal significant level for all the tests is set at $\alpha=0.05$, and all statistics are computed from $30000$ independent replications. One can see as expected that the type I error of $T_N^{(\mathrm{MSSC})}$ does indeed remain near 5\% as $M$ increases.

\begin{table}[!ht]

\caption{Sample type I error at 5\% \label{table:type_I_error}}
\centering
\begin{tabular}{lllrr}
\toprule
     &     &    &  $T_N^{(\mathrm{MSSC})}$  \\
N & B & M &        &       \\
\midrule
42   & 20  & 10  &  0.021 \\
316  & 100 & 50  &  0.031 \\
659  & 180 & 90  &  0.037 \\
1044 & 260 & 130 &  0.037 \\
1459 & 340 & 170 &  0.040 \\
1901 & 420 & 210 &  0.042 \\
5623  & 1000 & 500  &  0.048 \\
13374 & 2000 & 1000 &  0.051 \\
\bottomrule
\end{tabular}
\end{table}

\subsection{Power}
We now compare the power of our new test statistic against other independence test statistics which are designed to work in the high-dimensional regime. We define the Linear Spectral Statistic (LSS) test from \cite{loubaton2021large} for any $\epsilon > 0$ by
\begin{equation}
    \label{definition:T_LSS}
    T_N^{(\mathrm{LSS})}= \mathds{1}\left(\sup_{\nu\in[0,1]}\frac{\left|\frac{1}{M}\Tr f(\hat{\C}(\nu))-\int_\Rbb f\diff\mu_{MP}^{(c_N)}\right|}{N^\epsilon(B/N)}>\kappa_{1-\alpha}\right)
\end{equation}
where $\mu_{MP}^{(c)}$ represents the Marcenko-Pastur distribution with parameter $c$ defined by 
\[ 
    d\mu_{MP}^{(c)}(\lambda)= \left(1-\frac{1}{c}\right)_+\delta_0(\diff\lambda)  +  \frac{\sqrt{(\lambda_+-\lambda)(\lambda-\lambda_-)}}{2\pi c\lambda}\mathds{1}_{[\lambda_-,\lambda_+]}(\lambda)\diff\lambda 
\]
where $\lambda_\pm=(1\pm\sqrt{c})^2$, $(\cdot)_+:=\max(\cdot,0)$, $c_N:=\frac{M}{B+1}$ and $f$ is some function defined on $\Rbb_+$ satisfying regularity assumptions (see more details in \cite{loubaton2021large}). In practice, $\epsilon$ will be taken equal to $0.1$. It is proven in \cite{loubaton2021large} that under $\Hcal_0$, $T_N^{(\mathrm{LSS})} \rightarrow 0$ almost surely in the high-dimensional regime but the exact asymptotic distribution of the LSS test is unknown. Therefore, the detection threshold $\kappa_{1-\alpha}$ for this test is based on a sample quantile of $T_N^{(\mathrm{LSS})}$ under $\Hcal_0$ computed from Monte-Carlo simulation. For fairness comparison, we also use this procedure for the new test statistic $T_N^{(\mathrm{MSSC})}$. More precisely, we compute the sample $(1-\alpha)$--quantile $\kappa_{1-\alpha}$ of a test statistic $T_N^{(\mathrm{LSS})}$ from samples under $\Hcal_0$, and then reject the null hypothesis under $\Hcal_1$ if $T_N^{(LSS)}>\kappa_{1-\alpha}$. It remains to choose a test function $f$, and we again follow \cite{loubaton2021large} by considering 
\begin{itemize}
    \item the Frobenius test $T_N^{(\mathrm{FROB})}$ when $f(x)=(x-1)^2$
    \item the logdet test $T_N^{(\mathrm{LOG})}$ when $f(x)=\log x$
\end{itemize}
It remains to define the alternatives. For this, we consider the following multidimensional $AR(1)$ model:
\begin{equation}
\label{equation:state_space_model}
    \begin{array}{ll}
        \y_{n+1} = \A\y_n + \epsilonbs_n
    \end{array}
\end{equation}
where $(\epsilonbs_n)_{n \in \mathbb{Z}}$ is a sequence of independent $\Ncal_{\Cbb^M}(\bf{0},\I)$ distributed random vectors, and $\A$ is a bidiagonal matrix. Three choices of $\A$ ($\A^{(\Hcal_0)}$, $\A^{(\Hcal_{1,\mathrm{loc}})}$, $\A^{(\Hcal_{1,\mathrm{glob}})}$) allows us to define two alternatives:
\begin{enumerate}
    \item $\Hcal_0$: for $|\theta|<1$:
$$
\A^{(\Hcal_0)} = \left( \begin{array}{cccccc} \theta & 0 & \ldots & \ldots & \ldots & 0 \\
                                  0 & \theta & 0 & \ldots & \ldots & 0 \\
                                  0 & 0 & \theta & 0 & \ldots & 0 \\
                                  \vdots & \ddots & \ddots & \ddots & \ddots & \vdots \\
                                   \vdots & \ddots & \ddots & \ddots & \ddots & 0 \\
                                   0 & \ldots & \ldots & 0 & 0 & \theta 
                                   \end{array} \right)
$$     
so the signals $((y_{m,n})_{n\in\Zbb})_{m=1,\ldots,M}$ are mutually independent.
    \item $\Hcal_{1,\mathrm{loc}}$: for $|\theta|<1$ and $\beta\in\Rbb$:
$$
\A^{(\Hcal_{1,\mathrm{loc}})} = \left( \begin{array}{cccccc} \theta & 0 & \ldots & \ldots & \ldots & 0 \\
                                  \beta & \theta & 0 & \ldots & \ldots & 0 \\
                                  0 & 0 & \theta & 0 & \ldots & 0 \\
                                  \vdots & \ddots & \ddots & \ddots & \ddots & \vdots \\
                                   \vdots & \ddots & \ddots & \ddots & \ddots & 0 \\
                                   0 & \ldots & \ldots & 0 & 0 & \theta 
                                   \end{array} \right)
$$  
so the couple of time series (1,2) is the unique correlated pair of signals. 
    \item $\Hcal_{1,\mathrm{glob}}$: for $|\theta|<1$ and $\beta\in\Rbb$:
$$
\A^{(\Hcal_{1,\mathrm{glob}})} = \left( \begin{array}{cccccc} \theta & 0 & \ldots & \ldots & \ldots & 0 \\
                                  \beta & \theta & 0 & \ldots & \ldots & 0 \\
                                  0 & \beta & \theta & 0 & \ldots & 0 \\
                                  \vdots & \ddots & \ddots & \ddots & \ddots & \vdots \\
                                   \vdots & \ddots & \ddots & \ddots & \ddots & 0 \\
                                   0 & \ldots & \ldots & 0 & \beta & \theta 
                                   \end{array} \right)
$$     
so all the signals are mutually correlated.
\end{enumerate}

We now fix the value of the parameters involved under the three hypotheses. $\theta$ will always be taken equal to $0.5$. Under $\Hcal_{1,\mathrm{loc}}$, $\beta=0.1$. Concerning the alternative $\Hcal_{1,\mathrm{glob}}$, more care is required to choose $\beta$. Indeed, one can define a measure of total dependence as:  
$$ r:=\frac{\int\|\S(\nu)-\diag\S(\nu)\|_F^2\diff\nu}{\int\|\S(\nu)\|_F^2\diff\nu}=\frac{\sum_{u\in\Zbb}\|\R(u)-\diag\R(u)\|_F^2}{\sum_{u\in\Zbb}\|\R(u)\|_F^2}$$
where $\R(u):=\Ebb[\y_{n+u}\y_{n}^*]$, $\S(\nu) = \sum_{u \in \Zbb} \R(u) \erm^{-\irm 2 \pi u \nu}$ and $\diag$ denotes the diagonal part operator. Clearly, $r=0$ under $\Hcal_0$, and as $r>0$ increases, the $M$--dimensional time series become correlated. We also see that for any fixed value of $\beta$, $r$ is increasing with $M$. It is therefore more desirable to tune $\beta:=\beta(M)$ such that $r$ remains constant as $M$ increases. This will enable our tests to be compared against an alternative which does not become asymptotically trivial. 

The two alternatives $\Hcal_{1,\mathrm{loc}}$ and $\Hcal_{1,\mathrm{glob}}$ are useful to measure the performance of the independence tests under two different setups. Under $\Hcal_{1,\mathrm{loc}}$, each pair of time series are independent except the pair $(y_{1,n})_{n \in \Zbb}$,$(y_{2,n})_{n \in \Zbb}$, whereas under $\Hcal_{1,\mathrm{glob}}$ each time series has a small correlation with every other time series. 


On Table \ref{table:power_global} and Table \ref{table:power_local} are presented the sample powers when the type I error is fixed at $5\%$ and for the considered tests and the two alternatives. The asymptotic regime is the same as the one considered for Table \ref{table:type_I_error}: $\rho=0.7$ and $\frac{M}{B+1}=0.2$. All statistics are computed from 30000 independent replications. We observe that under $\Hcal_{1,\mathrm{glob}}$, with $r=0.01$, all the tests asymptotically detect the alternative, however with different performances. The LSS test statistics show better power which indicates that they may be more suited to detect alternative under $\Hcal_{1,\mathrm{glob}}$ than the MSSC test statistics. Under $\Hcal_{1,\mathrm{loc}}$ the results are opposite: the power of $T_N^{(\mathrm{MSSC})}$ rapidly increases to $1$ as $M$ increases. These results are not surprising since the MSSC test statistic is designed to detect peaks in the off-diagonal entries of $\hat{\C}(\nu)$ which is exactly the class of alternative considered in $\Hcal_{1,\mathrm{loc}}$. However, when the correlations are spread among all pairs of time series under $\Hcal_{1,\mathrm{glob}}$, the test statistics based on the global behaviour of the eigenvalues of $\hat{\C}(\nu)$ seem more relevant.

\begin{table}[!ht]
\caption{Power comparison under $\Hcal_1$ global, type I error = 5\% \label{table:power_global}}
\centering
\begin{tabular}{lllrrr}
\toprule
     &     &  &  $T_N^{(\mathrm{FROB})}$ &  $T_N^{(\mathrm{LOG})}$ &   $T_N^{(\mathrm{MSSC})}$ \\
N & M & B &                &             &       \\
\midrule
42   & 10  & 20  &          0.050 &       0.049 &  0.052 \\
316  & 50  & 100 &          0.036 &       0.042 &  0.067 \\
659  & 90  & 180 &          0.067 &       0.065 &  0.086 \\
1044 & 130 & 260 &          0.142 &       0.122 &  0.133 \\
1459 & 170 & 340 &          0.339 &       0.255 &  0.214 \\
1901 & 210 & 420 &          0.601 &       0.462 &  0.328 \\
2364 & 250 & 500 &          0.836 &       0.682 &  0.503 \\
2846 & 290 & 580 &          0.960 &       0.852 &  0.672 \\
\bottomrule
\end{tabular}
\end{table}

\begin{table}[!ht]
\caption{Power comparison under $\Hcal_1$ local, type I error = 5\% \label{table:power_local}}
\centering
\begin{tabular}{lllrrr}
\toprule
     &     &  &  $T_N^{(\mathrm{FROB})}$ &  $T_N^{(\mathrm{LOG})}$ &   $T_N^{(\mathrm{MSSC})}$ \\
N & M & B &                &             &       \\
\midrule
42   & 10  & 20  &          0.049 &       0.049 &  0.061 \\
316  & 50  & 100 &          0.038 &       0.044 &  0.352 \\
659  & 90  & 180 &          0.038 &       0.041 &  0.881 \\
1044 & 130 & 260 &          0.034 &       0.038 &  0.999 \\
1459 & 170 & 340 &          0.034 &       0.038 &  1.000 \\
1901 & 210 & 420 &          0.035 &       0.039 &  1.000 \\
2364 & 250 & 500 &          0.031 &       0.039 &  1.000 \\
2846 & 290 & 580 &          0.032 &       0.036 &  1.000 \\
\bottomrule
\end{tabular}

\end{table}

On Figure \ref{figure:roc} are represented the ROC for each test under both alternatives.
We observe that $T_N^{(\mathrm{FROB})}$ and $T_N^{(\mathrm{LOG})}$ have similar performance and outperform $T_N^{(\mathrm{MSSC})}$ for the alternative $\Hcal_{1,\mathrm{glob}}$, while $T_N^{(\mathrm{MSSC})}$ has better performance for $\Hcal_{1,\mathrm{loc}}$.

\begin{figure}[!ht]
\centering
     \includegraphics[width=1.0\textwidth]{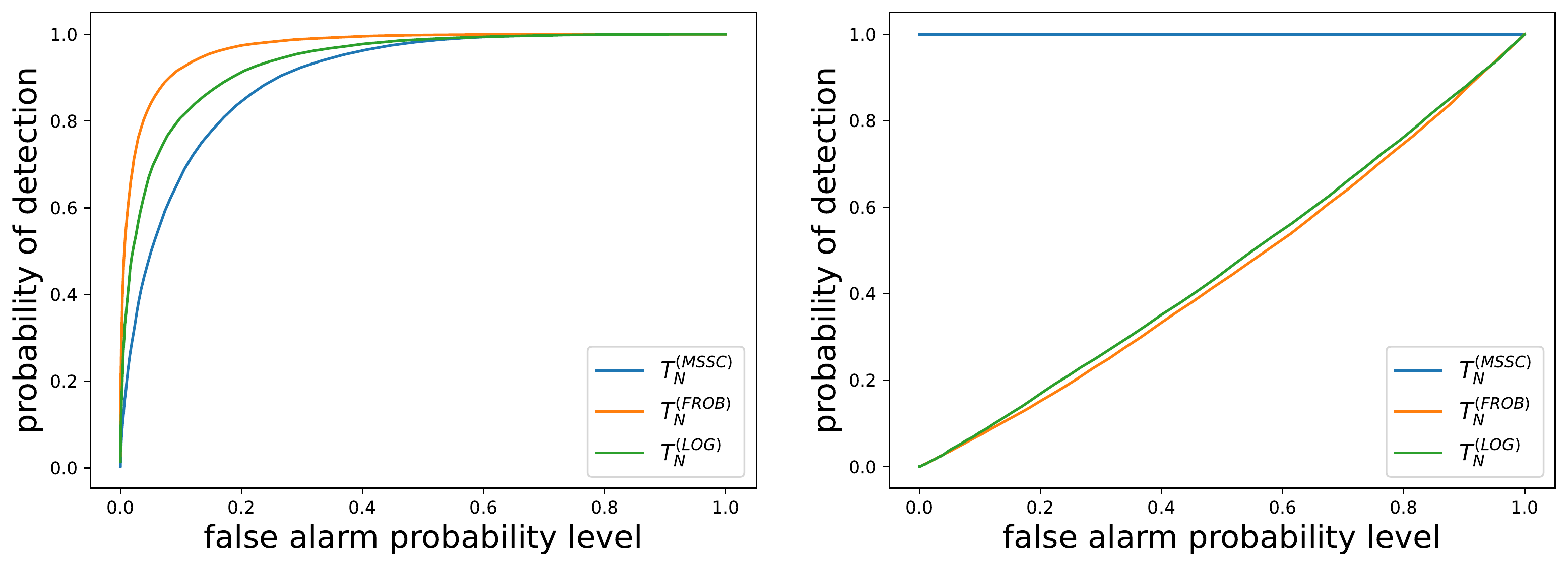}
      \caption{ROC associated to each test under $\Hcal_1^{(glob)}$ with $r=0.01$ (left) and $\Hcal_1^{(loc)}$ with $\beta=0.1$ (right) when $(N,M,B)=(2846,290,580)$)}
      \label{figure:roc}. 
\end{figure}

\section{Proof of Theorem \ref{theorem:main}}
\label{section:proof_main_theorem}
We will detail in this section the main steps to prove Theorem \ref{theorem:main}, while some details will be left in the Appendix.

\subsection{General approach}

First, we notice that the frequency smoothed estimate $\hat{s}_{i,j}(\nu)$ defined in \eqref{eq:sijhat} can be written as
\begin{equation}
  \label{equation:s_hat_sesquilinear}
    \hat{s}_{ij}(\nu) = \frac{1}{B+1} \xibs_{y_j}(\nu)^* \xibs_{y_i}(\nu)
\end{equation}
where \[
    \xibs_{y_i}(\nu) = \left(\xi_{y_i}\left(\nu-\frac{B}{2N}\right), \ldots, \xi_{y_i}\left(\nu+\frac{B}{2N}\right)\right)^T.
\]
This is a sesquilinear form of the finite Fourier transform of the $M$ time series samples $(y_{i,1},\ldots,y_{i,N})_{i\in[M]}$.
To handle the statistical dependence between the components of $\xibs_{y_i}(\nu)$, we use the well-known Bartlett decomposition (see for instance \cite{walker1965some}) whose procedure is described hereafter.

From Assumptions \ref{assumption:short_memory} and \ref{assumption:non_vanishing_spectrum}, the spectral distribution of $(y_{m,n})_{n \in \Zbb}$ is absolutely continuous with density $s_m$ being uniformly bounded and bounded away from $0$. Therefore, from Wold's Theorem \cite[Th. 5.7.1, Th. 5.7.2]{Brockwell2006},  each time series $(y_{m,n})_{n \in \Zbb}$ admits a causal and causally invertible linear representation  in terms of its normalized innovation sequence:
\begin{equation}
\label{equation:representation_lineaire_causale_y}
    y_{m,n} = \sum_{k=0}^{+\infty} a_{m,k} \epsilon_{m,n-k},
\end{equation}
where $(\epsilon_{1,k})_{k\in\Zbb},\ldots,(\epsilon_{M,k})_{k\in\Zbb}$ are mutually independent sequences of $\Ncal_\Cbb(0,1)$ i.i.d. random variables,
and $(a_{1,k})_{k \in \Nbb}, \ldots,(a_{M,k})_{k \in \Nbb} \in \ell^2(\Nbb)$ such that if
\begin{equation}                
\label{definition_h}
    h_m(\nu) = \sum_{k=0}^{+\infty} a_{m,k} e^{-2i\pi k\nu}
\end{equation}
then $|h_m(\nu)|^2=s_m(\nu)$ and $h_m(\nu)$ coincides with the outer causal spectral factor of $s_m(\nu)$. Define now  $\tilde{s}_{ij}(\nu)$, an approximation of $\hat{s}_{ij}(\nu)$, as:
\begin{equation*}
    \tilde{s}_{ij}(\nu) = \frac{1}{B+1}\sum_{b=-B/2}^{B/2}h_i\left(\nu+\frac{b}{N}\right)\overline{h_j\left(\nu+\frac{b}{N}\right)}\xi_{\epsilon_i}\left(\nu+\frac{b}{N}\right)\overline{\xi_{\epsilon_j}\left(\nu+\frac{b}{N}\right)}
\end{equation*}
or equivalently 
\begin{equation}
\label{equation:s_tilde_sesquilinear}
    \tilde{s}_{ij}(\nu) = \xibs_{\epsilon_j}(\nu)^* \frac{\Pibs_{ij}(\nu)}{B+1} \xibs_{\epsilon_i}(\nu)
\end{equation}
where 
\begin{equation}
\label{equation:definition_Pibs}
    \Pibs_{ij}(\nu) = \diag\left(h_i\left(\nu+\frac{b}{N}\right)\overline{h_j\left(\nu+\frac{b}{N}\right)}\right)_{b=-B/2,\ldots,B/2}
\end{equation}
and 
\[
  \xibs_{\epsilon_i}(\nu) = \left(\xi_{\epsilon_i}\left(\nu-\frac{B}{2N}\right), \ldots, \xi_{\epsilon_i}\left(\nu+\frac{B}{2N}\right)\right)^T.
\]
Instead of working directly with $|\hat{c}_{ij}(\nu)|^2=\frac{|\hat{s}_{ij}(\nu)|^2}{\hat{s}_{i}(\nu)\hat{s}_{j}(\nu)}$, it turns out that it is more convenient to show the limiting Gumbel distribution for $\frac{|\tilde{s}_{ij}(\nu)|^2}{\sigma_{ij}^2(\nu)}$ where
\begin{align} 
\label{equation:definition_sigma_ij}
\notag
    \sigma_{ij}^2(\nu) &= \frac{1}{B+1} \sum_{b=-B/2}^{B/2} \left|h_i\left(\nu+\frac{b}{N}\right)\right|^{2} \left|h_j\left(\nu+\frac{b}{N}\right)\right|^{2} \\ 
\notag
    &=  \frac{1}{B+1} \sum_{b=-B/2}^{B/2} s_i\left(\nu+\frac{b}{N}\right)s_j\left(\nu+\frac{b}{N}\right) \\
    &:= \frac{\Tr\Sigmabs_{ij}(\nu)}{B+1}
\end{align}
and where
\begin{multline} 
\label{equation:definition_Sigmabs_ij}
    \Sigmabs_{ij}(\nu) := \Pibs_{ij}^*(\nu)\Pibs_{ij}(\nu) \\ = \diag\left(\left|h_i\left(\nu+\frac{b}{N}\right)\right|^{2} \left|h_j\left(\nu+\frac{b}{N}\right)\right|^{2}, b=-\frac{B}{2},\ldots,\frac{B}{2}\right).
\end{multline}
This is the aim of Proposition \ref{proposition:s_tilde_hat_sigma} below.
\begin{proposition}[Gumbel limit for $\max_{(i,j,\nu)\in\Ical}|\tilde{s}_{ij}(\nu)|^2$]
\label{proposition:s_tilde_hat_sigma}
Under Assumptions \ref{assumption:gaussian_y_n} -- \ref{assumption:non_vanishing_spectrum}, for any $t\in\Rbb$, we have
\begin{multline}
\label{equation:s_tilde_hat_sigma}
    \Prob\left(\max_{(i,j,\nu)\in\Ical}(B+1)\frac{|\tilde{s}_{ij}(\nu)|^2}{\sigma_{ij}^2(\nu)} \le t + \log\frac{N}{B+1} + \log \frac{M(M-1)}{2} \right) \xrightarrow[N\to+\infty]{} e^{-e^{-t}}.
\end{multline}
\end{proposition}
Once equipped with Proposition \ref{proposition:s_tilde_hat_sigma}, it remains then to show that  $\max_{(i,j,\nu)\in\Ical}\frac{|\tilde{s}_{ij}(\nu)|^2}{\sigma_{ij}^2(\nu)}$ is close enough from $\max_{(i,j,\nu)\in\Ical}|\hat{c}_{ij}(\nu)|^2$ to prove that these quantities have the same limiting distribution. This result is given by the following Proposition.
\begin{proposition}[]
\label{proposition:same_limiting_distribution}
Under Assumptions \ref{assumption:gaussian_y_n} -- \ref{assumption:non_vanishing_spectrum}, as $N\to\infty$,
\[
    \max_{(i,j,\nu)\in\Ical}(B+1)\frac{|\tilde{s}_{ij}(\nu)|^2}{\sigma_{ij}^2(\nu)} - \max_{(i,j,\nu)\in\Ical}(B+1)|\hat{c}_{ij}(\nu)|^2 = o_P(1).
\]
\end{proposition}

As Theorem \ref{theorem:main} is directly obtained by Proposition \ref{proposition:s_tilde_hat_sigma}, Proposition \ref{proposition:same_limiting_distribution} and an application of Slutsky's lemma, the two remaining subsections are devoted to the proofs of Proposition \ref{proposition:s_tilde_hat_sigma} and Proposition \ref{proposition:same_limiting_distribution}. 

\subsection{Proof of Proposition \ref{proposition:s_tilde_hat_sigma}}

To prove Proposition \ref{proposition:s_tilde_hat_sigma}, the main tool is the Lemma A.4 from \cite{jiang2004asymptotic}, which is a special case of Poisson approximation from \cite{arratia1989two}. We rewrite it here for the sake of completeness.

\begin{lemma}
\label{lemma:Arratia_Goldstein_Gordon}
Let $(X_\alpha)_{\alpha\in\Ical}$ be a finite collection of Bernoulli random variables, and for each $\alpha\in\Ical$, let $\Ical_\alpha\subset\Ical$ such that $\alpha\in\Ical_\alpha$. Then,
\begin{equation*}
    \left|\Prob\left(\sum_{\alpha\in\Ical}X_\alpha=0\right) - \exp\left(-\sum_{\alpha\in\Ical}\Prob(X_\alpha=1)\right)\right| \le \Delta_1 + \Delta_2 + \Delta_3
\end{equation*}
where
\begin{align*}
    &\Delta_1 = \sum_{\alpha\in\Ical}\sum_{\beta\in \Ical_\alpha} \Prob\left(X_\alpha=1\right) \, \Prob\left(X_\beta=1\right) \\
    &\Delta_2 = \sum_{\alpha\in\Ical}\sum_{\beta\in \Ical_\alpha\setminus{\{\alpha\}}} \Prob\left(X_\alpha=1, \, X_\beta=1\right) \\
    &\Delta_3 = \sum_{\alpha\in \Ical}\Ebb\left|\Prob\left(X_\alpha=1 | \left(X_\beta\right)_{\beta\in \Ical\setminus\Ical_\alpha}\right)-\Prob(X_\alpha=1)\right|
\end{align*}
In particular, if for each $\alpha \in \Ical$, $X_\alpha$ is independent of $\{X_\beta:\beta\in\Ical\setminus\Ical_\alpha\}$, then $\Delta_3=0$.
\end{lemma}

Lemma \ref{lemma:Arratia_Goldstein_Gordon} is the keystone for the proof of Proposition \ref{proposition:s_tilde_hat_sigma}, and is a standard tool for analyzing distributions of maxima of dependent random variables. We now prove Proposition \ref{proposition:s_tilde_hat_sigma}.

\begin{proof} 
We start by proving \eqref{equation:s_tilde_hat_sigma}.  Define
\begin{equation}
\label{definition:t_N}
    t_N=\sqrt{x+\log\frac{M(M-1)}{2}+\log \frac{N}{B+1}}
\end{equation}
and for $(i,j,\nu)\in\Ical$ (recall that $\Ical$ is defined in \eqref{definition:Ical_N}, and that it depends on $N$, but in order to avoid cumbersome notations we do not recall this dependency) the Bernoulli random variables $X_{ij}(\nu)$ as
\begin{equation}
\label{equation:definition_X_alpha}
    X_{ij}(\nu) := \mathds{1}\left((B+1)\frac{|\tilde{s}_{ij}(\nu)|^2}{\sigma_{ij}^2(\nu)}>t_N^2\right).
\end{equation}
Define the set $\Ical_{(i,j,\nu)}$ 
\begin{equation}
\label{definition:Bcal_alpha}
    \Ical_{(i,j,\nu)} = \{(i',j',\nu): 1\le i'<j'\le M,  i=i' \text{ or } j=j'\}.
\end{equation}
From \eqref{equation:s_tilde_sesquilinear} and under Assumption \ref{assumption:gaussian_y_n}, if $(i',j',\nu') \in \Ical\backslash\Ical_{(i,j,\nu)}$, then $\tilde{s}_{i'j'}(\nu')$ is independent
from $\tilde{s}_{ij}(\nu)$ since we have either
\begin{enumerate}
\item[(1)] $i' \neq i$, $j' \neq j$, $\nu'=\nu$;
\item[(2)] $i'=i$ or $j'=j$, and $\nu' \neq \nu$ (implying $|\nu-\nu'| > \frac{B}{N}$ by assumption), in which case $\left(\xibs_{\epsilon_{i'}}(\nu'),\xibs_{\epsilon_{j'}}(\nu')\right)$ is independent
    from $\left(\xibs_{\epsilon_{i}}(\nu),\xibs_{\epsilon_{j}}(\nu)\right)$.
\end{enumerate}
From the definition of $X_{ij}(\nu)$ in \eqref{equation:definition_X_alpha},
\[
    \Prob\left((B+1)\max_{(i,j,\nu)\in\Ical}\frac{|\tilde{s}_{ij}(\nu)|^2}{\sigma^2_{ij}(\nu)}\le t_N^2\right) = \Prob\left(\sum_{(i,j,\nu)\in\Ical} X_{ij}(\nu) =0\right)
\]
which can be estimated by Lemma \ref{lemma:Arratia_Goldstein_Gordon} as:
\[ 
    \left|\Prob\left(\sum_{(i,j,\nu)\in\Ical} X_{ij}(\nu) =0\right) - e^{-\lambda}\right| \le \Delta_1 + \Delta_2 + \Delta_3
\]
where 
\[
    \lambda = \sum_{(i,j,\nu)\in\Ical} \Prob\left(X_{ij}(\nu)=1\right)  = \sum_{(i,j,\nu)\in\Ical}\Prob\left((B+1)\frac{|\tilde{s}_{ij}(\nu)|^2}{\sigma^2_{ij}(\nu)}>t_N^2\right)
\]
and 
\begin{align}
\label{eq:def_Deltabs_1}
  &\Delta_1
    = \sum_{(i,j,\nu)\in\Ical} \ \sum_{(i',j',\nu)\in\Ical_{(i,j,\nu)}}
    \Prob\left((B+1)\frac{|\tilde{s}_{i,j}(\nu)|^2}{\sigma^2_{ij}(\nu)}>t_N^2\right)
    \,
    \Prob\left((B+1)\frac{|\tilde{s}_{i'j'}(\nu)|^2}{\sigma^2_{i'j'}(\nu)}>t_N^2\right) \\
\label{eq:def_Deltabs_2}
  &\Delta_2 =
    \sum_{(i,j,\nu)\in\Ical} \ \sum_{\substack{(i',j',\nu)\in\Ical_{(i,j,\nu)} \\ (i',j') \neq (i,j)}}
    \Prob\left((B+1)\frac{|\tilde{s}_{i,j}(\nu)|^2}{\sigma^2_{ij}(\nu)}>t_N^2, \, (B+1)\frac{|\tilde{s}_{i'j'}(\nu)|^2}{\sigma^2_{i'j'}(\nu)}>t_N^2\right) 
\end{align}
\begin{multline}
\label{eq:def_Deltabs_3}
\Delta_3 =
\sum_{(i,j,\nu)\in\Ical}
\Ebb\left|\Prob\left((B+1)\frac{|\tilde{s}_{ij}(\nu)|^2}{\sigma^2_{ij}(\nu)}>t_N^2 | \left(\tilde{s}_{i'j'}(\nu')\right)_{(i',j',\nu') \in \Ical \backslash \Ical_{(i,j,\nu)}}\right) \right.
\\
\left. -\Prob\left((B+1)\frac{|\tilde{s}_{ij}(\nu)|^2}{\sigma^2_{ij}(\nu)}>t_N^2\right)\right|.
\end{multline}
We now have to control the four quantities $\lambda$, $\Delta_1$, $\Delta_2$ and $\Delta_3$, which requires studying moderate deviations results for
\[
    \Prob\left((B+1)\frac{|\tilde{s}_{ij}(\nu)|^2}{\sigma^2_{ij}(\nu)}>t_N^2\right)
\]
as well as
\[
    \Prob\left((B+1)\frac{|\tilde{s}_{ij}(\nu)|^2}{\sigma^2_{ij}(\nu)}>t_N^2, \, (B+1)\frac{|\tilde{s}_{i'j'}(\nu)|^2}{\sigma^2_{i'j'}(\nu)}>t_N^2\right)
  \]
  for all $(i',j',\nu) \in \Ical_{(i,j,\nu)}$.
The following Proposition \ref{proposition:deviation_modere}, proved in \ref{appendix:moderate_deviations}, provides exactly this. 
\begin{proposition}
\label{proposition:deviation_modere}
Under Assumptions \ref{assumption:gaussian_y_n} -- \ref{assumption:non_vanishing_spectrum}, there exists a constant $\eta>0$ such that for any $C>0$, we have
\begin{equation}
\label{equation:moderate_deviation_sij}
\max_{t\in\left[0,C B^{\eta}\right]}\max_{(i,j,\nu)\in\Ical}\left|\Prob\left((B+1)\frac{|\tilde{s}_{ij}(\nu)|^2}{\sigma^2_{ij}(\nu)}>t^2\right)e^{t^2} - 1\right|
\xrightarrow[N\to\infty]{} 0
\end{equation}
and 
\begin{align}
    &\max_{t,s\in\left[0, C B^{\eta}\right]}\max_{\substack{(i,j,\nu)\in\Ical \\ (i',j',\nu)\in\Ical_{(i,j,\nu)}}}
    \Biggl|
    \Prob\left((B+1)\frac{|\tilde{s}_{ij}(\nu)|^2}{\sigma^2_{ij}(\nu)}>t^2, (B+1)\frac{|\tilde{s}_{i'j'}(\nu)|^2}{\sigma^2_{i'j'}(\nu)}>s^2\right)
    \notag\\
    &\qquad\qquad\qquad\qquad\qquad\qquad\qquad\qquad
      \times e^{t^2+s^2} - 1
    \Biggr|
  \xrightarrow[N\to\infty]{} 0.
  \label{equation:joint_deviation_modere}
\end{align}
\end{proposition}
First, concerning $\exp(-\lambda)$, since $t_N$ as defined in \eqref{definition:t_N} is $\Ocal(\log N)$, one can use Proposition \ref{proposition:deviation_modere} to get
\begin{align}
  \exp\left(-\lambda\right)
  &= \exp\left(-\sum_{(i,j,\nu)\in\Ical}\Prob\left((B+1)\frac{|\tilde{s}_{ij}(\nu)|^2}{\sigma^2_{ij}(\nu)}>t_N^2\right)\right)
  \notag\\
  &= \exp\left(-\frac{N}{B+1}\frac{M(M-1)}{2} e^{-t_N^2}(1+o(1))\right)
  \notag\\
  & \xrightarrow[N\to\infty]{} \exp\left(-\exp(-x)\right).
    \notag
\end{align}
We now turn to the control of $\Delta_1$, $\Delta_2$, and $\Delta_3$. Regarding $\Delta_3$, since under Assumption \ref{assumption:gaussian_y_n} the random variables $\tilde{s}_{ij}(\nu)$
and $\tilde{s}_{i'j'}(\nu')$  for $(i',j',\nu') \in \Ical\backslash\Ical_{(i,j,\nu)}$ are independent, we clearly have $\Delta_3=0$.
Consider now \eqref{eq:def_Deltabs_1} and \eqref{eq:def_Deltabs_2}. The aim is to show that $\Delta_1=o(1)$ and $\Delta_2=o(1)$ when $t_N$ is defined by \eqref{definition:t_N}. Using the moderate deviation result \eqref{equation:moderate_deviation_sij} from Proposition \ref{proposition:deviation_modere}, and recalling that $C$ represents a universal constant independent of $N$ whose value can change from one line to another, we get:
\begin{align*}
  \Delta_1 &\le
             \underbrace{|\Ical|}_{\Ocal(\frac{N}{B}M^2)} \underbrace{\max_{(i,j,\nu) \in \Ical}|\Ical_{(i,j,\nu)}|}_{\Ocal(M)}
             \max_{(i,j,\nu)\in\Ical} \Prob\left((B+1)\frac{|\tilde{s}_{i,j}(\nu)|^2}{\sigma_{i,j}(\nu)^2}>t_N^2\right)^2 \\
    &\le C \frac{N}{B}M^3 \underbrace{e^{-2 t_N^2}}_{\Ocal(\frac{1}{M^4}\frac{B^2}{N^2})} \underbrace{\max_{(i,j,\nu)\in\Ical} \left(\Prob\left[(B+1)\frac{|\tilde{s}_{i,j}(\nu)|^2}{\sigma_{i,j}(\nu)^2}>t_N^2\right]e^{t_N^2}\right)^2}_{=1+o(1)} \\
    &= \Ocal\left(\frac{1}{N}\right).
\end{align*}
$\Delta_2$ is handled similarly with equation \eqref{equation:joint_deviation_modere} from Proposition \ref{proposition:deviation_modere}:
\begin{align*}
  \Delta_2 &=
             \sum_{(i,j,\nu)\in\Ical}\sum_{(i',j',\nu)\in \Ical_{(i,j,\nu)}}
             \Prob\left((B+1)\frac{|\tilde{s}_{ij}(\nu)|^2}{\sigma^2_{ij}(\nu)}>t_N^2, \, (B+1)\frac{|\tilde{s}_{i'j'}(\nu)|^2}{\sigma^2_{i'j'}(\nu)}>t_N^2\right)
  \\
           &\le
             |\Ical| \ \max_{(i,j,\nu)\in \Ical}|\Ical_{(i,j,\nu)}| \ e^{-2t_N^2}
  \\
           & \qquad \times  \underbrace{\max_{(i,j,\nu)\in\Ical}\max_{(i',j',\nu)\in\Ical_{(i,j,\nu)}} \Prob\left((B+1)\frac{|\tilde{s}_{i,j}(\nu)|^2}{\sigma_{i,j}(\nu)^2}>t_N^2,
             (B+1)\frac{|\tilde{s}_{i'j'}(\nu)|^2}{\sigma^2_{i'j'}(\nu)}>t_N^2\right) e^{2t_N^2}}_{=1+o(1)}
  \\
    &= \Ocal\left(\frac{1}{N}\right).
\end{align*}
The proof of \eqref{equation:s_tilde_hat_sigma} is complete. 
\end{proof}

\subsection{Proof of Proposition \ref{proposition:same_limiting_distribution}}

To prove Proposition \ref{proposition:same_limiting_distribution}, ie. the fact that $\max_{(i,j,\nu)\in\Ical}\frac{|\hat{s}_{ij}(\nu)|^2}{\hat{s}_i(\nu)\hat{s}_j(\nu)}$ and
$\max_{(i,j,\nu)\in\Ical}\frac{|\tilde{s}_{ij}(\nu)|^2}{\sigma_{ij}^2(\nu)}$ are close enough in probability, we work separately on the numerator and the denominator. This constitutes the statement of the two following propositions. 

\begin{proposition}[Change of numerator]
\label{proposition:controle_T}
Under Assumptions \ref{assumption:gaussian_y_n} -- \ref{assumption:non_vanishing_spectrum}, there exists $\delta>0$ such that as $N\to\infty$,
\begin{equation}
\label{equation:controle_T}
    \sqrt{B+1}\max_{(i,j,\nu)\in\Ical}\left|\hat{s}_{ij}(\nu) - \tilde{s}_{ij}(\nu)\right|=\Ocal_P(N^{-\delta}).
\end{equation}
\end{proposition}
The proof is deferred to \ref{appendix:proof_controle_T}. A consequence of Proposition \ref{proposition:controle_T} and Proposition \ref{proposition:s_tilde_hat_sigma} is that
\begin{align}
\label{equation:controle_sup_s_hat}
  &\sqrt{B+1}\max_{(i,j,\nu)\in\Ical}|\hat{s}_{ij}(\nu)|
  \notag\\
&\qquad\qquad\le
\sqrt{B+1}\max_{(i,j,\nu)\in\Ical}|\tilde{s}_{i,j}(\nu)|
+ \sqrt{B+1}\max_{(i,j,\nu)\in\Ical}\left|\hat{s}_{ij}(\nu) - \tilde{s}_{ij}(\nu)\right|
\notag\\
&\qquad\qquad= \Ocal_P\left(\sqrt{\log N}\right).
\end{align}

\begin{proposition}[Change of denominator]
\label{proposition:denominator}
Under Assumption \ref{assumption:short_memory}, for any $\epsilon>0$, as $N\to\infty$,
\begin{equation}
\label{equation:controle_s_hat_sigma_difference}
    \max_{(i,j,\nu)\in\Ical}\left|\hat{s}_i(\nu)\hat{s}_j(\nu) - \sigma_{ij}^2(\nu)\right| = \Ocal_P\left(\frac{B}{N}+\frac{N^\epsilon}{\sqrt{B}}\right).
\end{equation}
Moreover, 
\begin{equation}
\label{equation:bornes_sigma}
    0<\inf_{N\ge1}\min_{(i,j,\nu)\in\Ical}\sigma_{ij}^2(\nu) \le \sup_{N\ge1}\max_{(i,j,\nu)\in\Ical}\sigma_{ij}^2(\nu) < +\infty
\end{equation}
and 
\begin{equation}
\label{equation:borne_s_hat}
    \max_{i\in[M]}\max_{\nu\in\Gcal}\frac{1}{\hat{s}_i(\nu)}=O_P(1), \quad \max_{i\in[M]}\max_{\nu\in\Gcal}\hat{s}_i(\nu) = O_P(1).
\end{equation}
\end{proposition}
The proof is deferred to \ref{appendix:approx_s_hat_s}.
We recall that for any sequences $(a_n)$ and $(b_n)$, the following inequality holds:
\[
    \left|\sup_n a_n-\sup_n b_n\right| \le \sup_n |a_n - b_n|.
\]
Therefore, to show that Proposition  \ref{proposition:same_limiting_distribution} holds, it is enough to show that
\[
    \max_{(i,j,\nu)\in\Ical}\left|(B+1)\frac{|\tilde{s}_{ij}(\nu)|^2}{\sigma_{ij}^2(\nu)} - (B+1)|\hat{c}_{ij}(\nu)|^2\right| = o_P(1).
\]
This result could be proved by writing the following decomposition:
\[
    (B+1)\max_{(i,j,\nu)\in\Ical}\left|\frac{|\tilde{s}_{ij}(\nu)|^2}{\sigma_{ij}^2(\nu)} - \frac{|\hat{s}_{ij}(\nu)|^2}{\hat{s}_i(\nu)\hat{s}_j(\nu)}\right| \le \Psi_3(\Psi_1+\Psi_2).
\]
where
\begin{align*}
    \Psi_1 &:= (B+1)\max_{(i,j,\nu)\in\Ical}\left||\hat{s}_{ij}(\nu)|^2-|\tilde{s}_{ij}(\nu)|^2 \right| \hat{s}_i(\nu)\hat{s}_j(\nu) \\
    \Psi_2 &:= (B+1)\max_{(i,j,\nu)\in\Ical}|\hat{s}_{ij}(\nu)|^2|\hat{s}_i(\nu)\hat{s}_j(\nu)-\sigma_{ij}^2(\nu)| \\
    \Psi_3 &:= \max_{(i,j,\nu)\in\Ical} \frac{1}{\hat{s}_i(\nu)\hat{s}_j(\nu)\sigma_{ij}^2(\nu)}.
\end{align*}
It is clear by \eqref{equation:s_tilde_hat_sigma} that
\[
    \max_{(i,j,\nu)\in\Ical}(B+1)|\tilde{s}_{ij}(\nu)|^2 = \Ocal_P\left(\log N\right).
\]
Combining this with Proposition \ref{proposition:denominator} and equation \eqref{equation:controle_T} from Proposition \ref{proposition:controle_T}, there exists $\delta>0$ such that
\begin{multline*}
  (B+1)\max_{(i,j,\nu)\in\Ical}\left||\hat{s}_{ij}(\nu)|^2-|\tilde{s}_{ij}(\nu)|^2\right| \le
  \\
  \underbrace{\sqrt{B+1}\max_{(i,j,\nu)\in\Ical} (|\hat{s}_{ij}(\nu)|+|\tilde{s}_{ij}(\nu)|)}_{=\Ocal_P(\sqrt{\log N})} \\ \times \underbrace{\sqrt{B+1}\max_{(i,j,\nu)\in\Ical}\left||\hat{s}_{ij}(\nu)|-|\tilde{s}_{ij}(\nu)|\right|}_{=\Ocal_P(N^{-\delta})}
\end{multline*}
which is $\Ocal_P(\sqrt{\log N}N^{-\delta})$. Using \eqref{equation:borne_s_hat}, this implies that
\[
    \Psi_1 = \Ocal_P\left(\sqrt{\log N}N^{-\delta}\right).
\]
Similarly, using Proposition \ref{proposition:denominator}, for any $\epsilon>0$, 
\begin{align*}
    \Psi_2 &= \Ocal_P\left(\log N \left(\frac{B}{N}+\frac{N^\epsilon}{\sqrt{B}}\right)\right)  \\
    \Psi_3 &= \Ocal_P(1).
\end{align*}
Combining the estimates of $\Psi_1$, $\Psi_2$ and $\Psi_3$ we get that for any $\epsilon>0$:
\begin{multline*}
    (B+1)\max_{(i,j,\nu)\in\Ical}\left|\frac{|\tilde{s}_{ij}(\nu)|^2}{\sigma_{ij}^2(\nu)} - \frac{|\hat{s}_{ij}(\nu)|^2}{\hat{s}_i(\nu)\hat{s}_j(\nu)}\right| = \\ \Ocal_P\left(N^{-\delta} \sqrt{\log N}  + \log N \left(\frac{B}{N}+\frac{N^\epsilon}{\sqrt{B}}\right)\right).
\end{multline*}
This quantity is $o_P(1)$ if $\frac{N^\epsilon}{\sqrt{B}}=o(1)$ which is satisfied by choosing $\epsilon<\frac{\rho}{2}$ from Assumption \eqref{assumption:rate_NBM}.

\appendix

\section{Proof of Proposition \ref{proposition:denominator}}
\label{appendix:approx_s_hat_s}
Before proving \eqref{equation:controle_s_hat_sigma_difference}, the main result of Proposition \ref{proposition:denominator}, we focus first on proving \eqref{equation:bornes_sigma} and \eqref{equation:borne_s_hat}. Concerning \eqref{equation:bornes_sigma}, recall that $\sigma_{ij}^2(\nu)$ defined in \eqref{equation:definition_sigma_ij} is equal to:

\begin{equation}
\label{equation:ecriture_sigma}
    \sigma_{ij}^2(\nu) = \frac{1}{B+1} \sum_{b=-B/2}^{B/2} s_i\left(\nu+\frac{b}{N}\right)s_j\left(\nu+\frac{b}{N}\right).
\end{equation}
By Assumption \ref{assumption:short_memory}, it is clear that \eqref{equation:bornes_sigma} holds. We now focus on proving \eqref{equation:borne_s_hat}. Since by Assumption \ref{assumption:short_memory} and Assumption \ref{assumption:non_vanishing_spectrum} the true spectral densities $s_i(\nu)$ are far from $0$ and $+\infty$, the same result should also hold for the estimators $\hat{s}_i(\nu)$. More precisely, we prove the following lemma. 

\begin{lemma}
\label{lemma:deviation_s_s_hat}
Under Assumption \ref{assumption:short_memory},
\begin{equation}
\label{equation:biais_uniformity}
    \max_{i\in[M]}\max_{\nu\in\Fcal}|\Ebb\hat{s}_i(\nu) - s_i(\nu)| = \Ocal\left(\frac{B}{N}\right)
\end{equation}

Moreover, under Assumption \ref{assumption:gaussian_y_n} and Assumption \ref{assumption:short_memory}, for any $\epsilon>0$, there exist $\gamma>0$ and $N_0(\epsilon)\in\Nbb$ such that:
\begin{equation}
\label{equation:variance_uniformity}
    \Prob\left(\max_{i\in[M]}\max_{\nu\in\Fcal}|\hat{s}_i(\nu) - \Ebb[\hat{s}_i(\nu)]|>N^\epsilon \frac{1}{\sqrt{B}}\right) \le \exp\left(-N^\gamma\right)
\end{equation}
for $N>N_0(\epsilon)$.
\end{lemma}

\begin{proof}
These results are close to those proved in Lemma A.2 and Lemma A.3. from \cite{loubaton2021large}. We will therefore closely follow their proofs. We start with the bias. By the definition \eqref{equation:s_hat_sesquilinear} of $\hat{s}_i(\nu)$:
\[
    |\Ebb\hat{s}_i(\nu) - s_i(\nu)| = \left|\frac{1}{B+1}\sum_{b=-B/2}^{B/2}\Ebb\left|\xi_{y_i}\left(\nu+\frac{b}{N}\right)\right|^2-s_i(\nu)\right|.
\]
Inserting $s_i(\nu+\frac{b}{N})$, one can write:
\begin{multline*}
    |\Ebb\hat{s}_i(\nu) - s_i(\nu)| \le \left|\frac{1}{B+1}\sum_{b=-B/2}^{B/2}\left(\Ebb\left|\xi_{y_i}\left(\nu+\frac{b}{N}\right)\right|^2-s_i\left(\nu+\frac{b}{N}\right)\right)\right| \\ + \left|\frac{1}{B+1}\sum_{b=-B/2}^{B/2}\left(s_i\left(\nu+\frac{b}{N}\right)-s_i(\nu)\right)\right|.
\end{multline*}

\cite[Lemma A.1]{loubaton2021large} provides the following control for the first term of the right-hand side under Assumption \ref{assumption:short_memory}:
\begin{equation}
\label{equation:biais_s_i}
    \max_{\nu\in\Fcal}\max_{i\in[M]}\left|\frac{1}{B+1}\sum_{b=-B/2}^{B/2}\left(\Ebb\left|\xi_{y_i}\left(\nu+\frac{b}{N}\right)\right|^2-s_i\left(\nu+\frac{b}{N}\right)\right)\right| = \Ocal\left(\frac{1}{N}\right).
\end{equation}
Moreover, by Assumption \ref{assumption:short_memory}, a Taylor expansion of $s_i$ around $\nu+\frac{b}{N}$, provides the existence of a quantity $\nu_b$ such that:
$$s_i\left(\nu+\frac{b}{N}\right)=s_i(\nu)+\frac{b}{N}s_i'(\nu_b)$$
where by Assumption \ref{assumption:short_memory}, $\sup_{i\ge1}\sup_{\nu\in[0, 1]}|s_i'(\nu)|<+\infty$. Therefore, it holds that, uniformly in $\nu\in\Fcal$ and $i\in[M]$:
\begin{multline}
\label{eq:biais_s_i_hat}
\max_{\nu\in\Fcal}\max_{i\in[M]}\left|\frac{1}{B+1}\sum_{b=-B/2}^{B/2}\left(s_i\left(\nu+\frac{b}{N}\right)-s_i(\nu)\right)\right|
\\
= \max_{\nu\in\Fcal}\max_{i\in[M]}\left|\frac{1}{B+1}\sum_{b=-B/2}^{B/2}\frac{b}{N}s_i'(\nu_b)\right| = \Ocal\left(\frac{B}{N}\right).
\end{multline}

Combining the estimations \eqref{equation:biais_s_i} and \eqref{eq:biais_s_i_hat}, one get:
\[
    \max_{i\in[M]}\max_{\nu\in\Fcal}|\Ebb\hat{s}_i(\nu) - s_i(\nu)| = \Ocal\left(\frac{1}{N}+\frac{B}{N}\right)=\Ocal\left(\frac{B}{N}\right)
\]
which is the desired result.

The second part of the lemma is an extension of a similar result also proved in \cite[Lemma A.3]{loubaton2021large} (see also similar results in \cite{bentkus1983distribution}). Under Assumption \ref{assumption:gaussian_y_n} and Assumption \ref{assumption:short_memory}, they have shown that for any $\nu\in[0,1]$ and for any $\epsilon>0$, there exists $\gamma>0$ such that:
\[
    \Prob\left(\max_{i\in[M]}|\hat{s}_i(\nu) - \Ebb[\hat{s}_i(\nu)]|> \frac{N^\epsilon}{\sqrt{B}}\right) \le \exp-N^\gamma
\]
for large enough $N>N_0(\epsilon)$. It remains to extend this concentration result to handle the uniformity over $\nu\in\Fcal$. This is done easily by the union bound.
\end{proof}

We can now prove \eqref{equation:borne_s_hat}. For any $A>0$, inserting $\Ebb[\hat{s}_i(\nu)]$ and $s_i(\nu)$ we can write:
\begin{align*}
  &\Prob\left(\max_{i\in[M]}\max_{\nu\in\Fcal}\hat{s}_i(\nu)>A\right)
  \\
  &=  \Prob\left(\max_{i\in[M]}\max_{\nu\in\Fcal}\left|\hat{s}_i(\nu)-\Ebb[\hat{s}_i(\nu)]+\Ebb[\hat{s}_i(\nu)]-s_i(\nu)+s_i(\nu)\right|>A\right)
  \\
  &\leq
    \Prob\left(\max_{i\in[M]}\max_{\nu\in\Fcal}\left|\hat{s}_i(\nu)-\Ebb[\hat{s}_i(\nu)]\right) \right.
  \\
  &\qquad\qquad\qquad\left.  >  A-\max_{i\in[M]}\max_{\nu\in\Fcal}s_i(\nu) - \max_{i\in[M]}\max_{\nu\in\Fcal}\left|\Ebb[\hat{s}_i(\nu)]-s_i(\nu)\right|\right).
\end{align*}
By Lemma \ref{lemma:deviation_s_s_hat} equation \eqref{equation:biais_uniformity} and Assumption \ref{assumption:short_memory}, for $N$ large enough:
\[
    A-\max_{i\in[M]}\max_{\nu\in\Fcal}s_i(\nu) - \max_{i\in[M]}\max_{\nu\in\Fcal}\left|\Ebb[\hat{s}_i(\nu)]-s_i(\nu)\right| \ge \frac{A}{2}.
\] 
The deviation result \eqref{equation:variance_uniformity} from Lemma \ref{lemma:approx_s_hat_s} eventually provides:
\begin{multline*}
    \Prob\left(\max_{i\in[M]}\max_{\nu\in\Fcal}\hat{s}_i(\nu)>A\right) \\ \le \Prob\left(\max_{i\in[M]}\max_{\nu\in\Fcal}\left|\hat{s}_i(\nu)-\Ebb[\hat{s}_i(\nu)]\right|>\frac{A}{2}\right) \xrightarrow[N\to+\infty]{} 0.
\end{multline*}
The proof that $\max_{i\in[M]}\max_{\nu\in\Fcal}\frac{1}{\hat{s}_i(\nu)}=\Ocal_P(1)$ is done similarly by considering
\[
    \Prob\left(\max_{i\in[M]}\max_{\nu\in\Fcal}\frac{1}{\hat{s}_i(\nu)}>A\right).
\]
We now focus on \eqref{equation:controle_s_hat_sigma_difference}, and consider the following decomposition:
\begin{multline*}
  \max_{(i,j,\nu)\in\Ical}\left|\hat{s}_i(\nu)\hat{s}_j(\nu) - \sigma_{ij}^2(\nu)\right|
  \le
  \max_{(i,j,\nu)\in\Ical}\left|\hat{s}_i(\nu)\hat{s}_j(\nu) - s_i(\nu)s_j(\nu)\right| +
  \\ \max_{(i,j,\nu)\in\Ical}\left|s_i(\nu)s_j(\nu) - \sigma_{ij}^2(\nu)\right|.
\end{multline*}
The following two lemmas bound each term of the right hand side, and lead to \eqref{equation:controle_s_hat_sigma_difference}.

\begin{lemma}
\label{lemma:approx_s_hat_s}
Under Assumption \ref{assumption:short_memory}, for any $\epsilon>0$, as $N\to\infty$,
\[
    \max_{(i,j,\nu)\in\Ical}\left|\hat{s}_i(\nu)\hat{s}_j(\nu) - s_i(\nu)s_j(\nu)\right| = \Ocal_P\left(\frac{B}{N}+\frac{N^\epsilon}{\sqrt{B}}\right).
\]
\end{lemma}
\begin{proof}
Write
\begin{multline*}
    \max_{(i,j,\nu)\in\Ical}\left|\hat{s}_i(\nu)\hat{s}_j(\nu) - s_i(\nu)s_j(\nu)\right| \\ \le \underbrace{\max_{i\in[M]}\max_{\nu\in\Fcal}|\hat{s}_i(\nu)-s_i(\nu)|}_{=\Ocal_P(\frac{B}{N}+\frac{N^\epsilon}{\sqrt{B}})}\underbrace{\max_{j\in[M]}\max_{\nu\in\Fcal}\hat{s}_j(\nu)}_{=\Ocal_P(1)} \\ + \underbrace{\max_{i\in[M]}\max_{\nu\in\Fcal}s_i(\nu)}_{=\Ocal(1)}\underbrace{\max_{j\in[M]}\max_{\nu\in\Fcal}|\hat{s}_j(\nu)-s_j(\nu)|}_{=\Ocal_P(\frac{B}{N}+\frac{N^\epsilon}{\sqrt{B}})}
\end{multline*}
for any $\epsilon>0$, where each estimate comes from Lemma \ref{lemma:deviation_s_s_hat} and Assumption \ref{assumption:short_memory}. 
\end{proof}

\begin{lemma}
\label{lemma:change_normalisation}
Under Assumption \ref{assumption:short_memory}, as $N\to\infty$,
\[
    \max_{(i,j,\nu)\in\Ical}\left|\sigma^2_{ij}(\nu) - s_i(\nu)s_j(\nu)\right| = O\left(\frac{B}{N}\right).
\]
\end{lemma}

\begin{proof}
By Assumption \ref{assumption:short_memory}, the applications $\nu\mapsto s_i(\nu) $ are $C^1$, so by Taylor expansion of $s_i$ around $\nu+\frac{b}{N}$, there exist frequencies $\nu_{i,b}\in[\nu,\nu+b/N]$ such that: 
\begin{align*}
    s_i\left(\nu+\frac{b}{N}\right) = s_i(\nu)+\frac{b}{N}s_i'(\nu_{i,b})
\end{align*}
where $s_i(\nu)$ and $s_i'(\nu_{i,b})$ satisfies:
\[
    \sup_{i\ge1}\max_{\nu\in[0,1]}s_i(\nu)<+\infty, \quad \sup_{i\ge1}\max_{\nu\in[0,1]}|s'_i(\nu)|<+\infty
\]
Recall the expression \eqref{equation:ecriture_sigma} of $\sigma_{ij}^2(\nu)$, and write:
\begin{multline*}
    \left|\sigma_{ij}^2(\nu) - s_i(\nu)s_j(\nu)\right| 
    \\ = \frac{1}{B+1}\sum_{b=-B/2}^{B/2} \left(\underbrace{\left(s_i\left(\nu+\frac{b}{N}\right)-s_i(\nu)\right)}_{=\Ocal(B/N)}\underbrace{s_j\left(\nu+\frac{b}{N}\right)}_{=\Ocal(1)} + \right. \\ \left. \underbrace{s_i(\nu)}_{=\Ocal(1)}\underbrace{\left(s_j\left(\nu+\frac{b}{N}\right) - s_j(\nu)\right)}_{=\Ocal(B/N)} \right)
\end{multline*}
where each bound above is uniform over $(i,j,\nu)\in\Ical$.
\end{proof}

\section{Proof of Proposition \ref{proposition:controle_T}}
\label{appendix:proof_controle_T}
To prove Proposition \ref{proposition:controle_T}, we need the three following lemmas (Lemma \ref{lemma:controle_I}, Lemma \ref{lemma:lemma_1_walker} and Lemma \ref{lemma:controle_R}), which are exactly or slight modifications of results from \cite{walker1965some}. We recall that according to \eqref{equation:s_tilde_sesquilinear}, $\tilde{s}_{ij}(\nu)$ can be expressed as the following sesquilinear form 
\begin{align*}
  \tilde{s}_{ij}(\nu) = \xibs_{\epsilon_j}(\nu)^* \frac{\Pibs_{ij}(\nu)}{B+1} \xibs_{\epsilon_i}(\nu),
\end{align*}
where the random variables $\left(\epsilon_{j,n}\right)_{\substack{j \in [M] \\ n \in [N]}}$ are independent and indentically distributed as $\Ncal_{\Cbb}(0,1)$. For the remainder, we denote for all $j \in [M]$
by $I_{N,\epsilon_j}(\nu)$ the periodogram of $\left(\epsilon_{j,n}\right)_{n \in [N]}$ at frequency $\nu$, i.e.
\begin{align*}
    I_{N,\epsilon_j}(\nu) = \left|\xi_{\epsilon_j}(\nu)\right|^2.
\end{align*}
The two following lemmas provide controls for the maximum of $I_{N,\epsilon_j}(\nu)$ over $\nu$ and $j$.
\begin{lemma}
\label{lemma:controle_I}
It holds that
\begin{equation}
\label{equation:controle_I}
    \Ebb\left[\max_{j\in[M]}\max_{\nu\in\Fcal}I_{N,\epsilon_j}(\nu)\right] = \Ocal\left(\log N + \log M\right).
\end{equation}
\end{lemma}

\begin{proof}
By independence and Gaussianity of the observations from the time series $\epsilon_j$, it is well known that the random variables $(I_{N,\epsilon_j}(\nu))$ for $\nu\in\Fcal$ and $j\ge 1$ are independent exponential $\Ecal(1)$ random variables. Therefore, for any $x\ge0$: 
\[
    \Prob\left(\max_{j\in[M]}\max_{\nu\in\Fcal}I_{N,\epsilon_j}(\nu)\le x\right) = (1-e^{-x})^{MN}.
\]  
Using the change of variable $y=1-e^{-x}$: 
\begin{align*}
  \Ebb\left[\max_{j\in[M]}\max_{\nu\in\Fcal}I_{N,\epsilon_j}(\nu)\right]
  &= \int_0^{+\infty}\Prob\left(\max_{j\in[M]}\max_{\nu\in\Fcal}I_{N,\epsilon_j}(\nu)>x\right)\diff x \\ 
  &= \int_0^{+\infty}(1-(1-e^{-x})^{MN}) \diff x \\
  &= \int_0^1 \frac{1-y^{MN}}{1-y} \diff y \\
  &= \sum_{r=0}^{MN-1}\frac{1}{r+1} \\
  &= \Ocal \left(\log M + \log  N\right).
\end{align*}
This proves \eqref{equation:controle_I}. 
\end{proof}

Under Assumption \ref{assumption:rate_NBM}, \eqref{equation:controle_I} simply becomes:
\begin{equation}
\label{equation:controle_I_2}
    \Ebb\left[\max_{j\in[M]}\max_{\nu\in\Fcal}I_{N,\epsilon_j}(\nu)\right] = \Ocal\left(\log N\right)
\end{equation}

The following lemma is from \cite[Lemma 1]{walker1965some} that we rewrite here for the sake of completeness. It allows to extend a control from
$\max_{\nu\in\Fcal}I_{N,\epsilon_j}(\nu)$ to $\max_{\nu\in[0,1]}I_{N,\epsilon_j}(\nu)$.
\begin{lemma}
\label{lemma:lemma_1_walker}
There exists a universal constant $C$ such that:
\begin{equation}
\label{equation:lemma_1_walker}
    \max_{j \in [M]}\max_{\nu\in[0,1]}I_{N,\epsilon_j}(\nu) \le C \log N \max_{j \in [M]}\max_{\nu\in\Fcal}I_{N,\epsilon_j}(\nu).
\end{equation}
\end{lemma}

A direct consequence of Lemma \ref{lemma:lemma_1_walker} used in Lemma \ref{lemma:controle_I} is that
\begin{equation}
\label{equation:controle_I_continu}
    \Ebb\left[\max_{j\in[M]}\max_{\nu\in[0,1]}I_{N,\epsilon_j}(\nu)\right] = \Ocal\left(\log N(\log N + \log M)\right).
\end{equation}

The main argument in the proof of Proposition \ref{proposition:controle_T} is the following result.
\begin{lemma}
\label{lemma:controle_R}
Define
\begin{equation}
\label{equation:definition_R_j}
    R_{N,j}(\nu) = \left|\xi_{y_j}(\nu) - h_j(\nu)\xi_{\epsilon_j}(\nu)\right|.
\end{equation}
Under Assumptions \ref{assumption:gaussian_y_n}--\ref{assumption:short_memory}, for any $0<\delta<\frac{1}{2}$,
\begin{equation}
    \label{equation:controle_R}
    \max_{\nu\in\Fcal}\max_{j\in[M]} R_{N,j}(\nu) = \Ocal_P(N^{-\delta}).
\end{equation}
\end{lemma}

\begin{proof}
We closely follow the proof of Theorem 2b from \cite{walker1965some}. To prove \eqref{equation:controle_R}, the Markov inequality shows that it is sufficient to prove that for any $\delta<1/2$,
\[
    \Ebb\left[\max_{\nu\in\Fcal}\max_{j\in[M]} R_{N,j}(\nu)\right] = \Ocal(N^{-\delta}).
\]

We use the linear causal representation \eqref{equation:representation_lineaire_causale_y} of $y_j$ to write
\begin{align*}
    \xi_{y_j}(\nu) &= \frac{1}{\sqrt{N}}\sum_{n=1}^N y_{j,n}e^{-2i\pi (n-1)\nu} \\
    &= \frac{1}{\sqrt{N}}\sum_{n=1}^N \left(\sum_{u=0}^{+\infty} a_{j,u}\epsilon_{j,n-u}\right) e^{-2i\pi (n-1)\nu}.
\end{align*}
Since almost surely, for all $j \in [M]$, $n \in [N]$, $\left(a_{j,u}\epsilon_{j,n-u}\right)_{u \geq 0} \in \ell^2(\Nbb)$, we can switch the order of summation and make the change of variable $v = n-u$ to get
\begin{align*}
  \xi_{y_j}(\nu)
  = \frac{1}{\sqrt{N}}\sum_{u=0}^{+\infty}a_{j,u}e^{-2i\pi u\nu}\sum_{v=1-u}^{N-u}  \epsilon_{j,v}e^{-2i\pi (v-1)\nu}.
\end{align*}
Define
\begin{equation}
\label{equation:definition_Z}
    Z_{N,j,u}(\nu) = \left(\sum_{v=1-u}^{N-u} - \sum_{v=1}^{N}\right) \epsilon_{j,v}e^{-2i\pi (v-1)\nu}
\end{equation}
so that $R_{N,j}(\nu)$ can be rewritten as:
\begin{equation}
\label{equation:reecriture_R}
    R_{N,j}(\nu) = \left|\frac{1}{\sqrt{N}}\sum_{u=0}^{+\infty}a_{j,u}e^{-2i\pi u\nu} Z_{N,j,u}(\nu)\right|
\end{equation}
on which one can take the supremum over $j\in[M]$ and $\nu\in\Fcal$ on each side and arrive at the following inequality:
\begin{equation*}
    \max_{j\in[M]}\max_{\nu\in\Fcal}R_{N,j}(\nu) \le \frac{1}{\sqrt{N}}\max_{j\in[M]}\sum_{u=0}^{+\infty}|a_{j,u}| \max_{\nu\in\Fcal}|Z_{N,j,u}(\nu)| 
\end{equation*}
where the right hand side is also bounded by:
\begin{equation*}
    \max_{j\in[M]}\max_{\nu\in\Fcal}R_{N,j}(\nu) \le \frac{1}{\sqrt{N}}\max_{j_1\in[M]}\sum_{u=0}^{+\infty}|a_{j_1,u}|\max_{j_2\in[M], \nu\in\Fcal}|Z_{N,j_2,u}(\nu)|.
\end{equation*}
Note that for $u=0$, $\max_{j_2\in[M], \nu\in\Fcal}|Z_{N,j_2,u}(\nu)|=0$, so the sum in fact can be written as starting from $1$. For any $\gamma<1$, the Cauchy-Schwarz inequality provides:
\begin{multline*}
    \frac{1}{\sqrt{N}}\max_{j_1\in[M]}\sum_{u=1}^{+\infty}|a_{j_1,u}|\max_{j_2\in[M], \nu\in\Fcal}|Z_{N,j_2,u}(\nu)| \\ \le \frac{1}{\sqrt{N}}\sqrt{\max_{j_1\in[M]}\sum_{u=1}^{+\infty}u^{2\gamma}|a_{j_1,u}|^2} \sqrt{\sum_{u=1}^{+\infty}\frac{1}{u^{2\gamma}}\max_{j_2\in[M], \nu\in\Fcal_N}|Z_{N,j_2,u}(\nu)|^2}.
\end{multline*}
Taking the expectation (and an application of the Jensen inequality to exchange the expectation and the square root), we get the following bound:
\begin{multline}
\label{equation:borne_R_1}
    \Ebb\left[\max_{j\in[M]}\max_{\nu\in\Fcal}R_{N,j}(\nu)\right] \\ \le \frac{1}{\sqrt{N}}\sqrt{\max_{j_1\in[M]}\sum_{u=1}^{+\infty}u^{2\gamma}|a_{j_1,u}|^2} \sqrt{\Ebb\left[\sum_{u=1}^{+\infty}\frac{1}{u^{2\gamma}}\max_{j_2\in[M], \nu\in\Fcal}|Z_{N,j_2,u}(\nu)|^2\right]}.
\end{multline}
Consider the first term in the right hand side of \eqref{equation:borne_R_1}. We see that we need to transfer the uniform sumability property of the sequences $(r_{j,u})_{u\in\Nbb}$ from Assumption \ref{assumption:short_memory} to a sumability property on the sequences $(a_{j,u})_{u\in\Nbb}$ uniformly over the $j\ge1$ times series. Hopefully, Lemma D.1 from \cite{loubaton2020asymptotic}, a generalization of the Wiener-Lévy theorem, provides an answer that we rewrite here for sake of completeness. 

\begin{lemma}(Lemma D.1, \cite{loubaton2020asymptotic})
\label{lemma:loubaton_mestre}
Consider a function $F(z)$ holomorphic in a neighbourhood of the interval $[s_{\min}, s_{\max}]$ where $s_{\min}$ and $s_{\max}$ are defined by \eqref{equation:definition_smin_smax}. Under Assumption \ref{assumption:short_memory}, for each $\gamma<1$,
\[
    \sup_{j\ge 1} \sum_{u\in\Zbb} (1+|u|)^\gamma \left|\int_0^1 (F\circ s_j)(\nu) e^{2i\pi\nu u}  \diff u \right| <+\infty.
\]
\end{lemma}
We now show how Lemma \ref{lemma:loubaton_mestre} can be used to find a sumability property on the sequences $(a_{j,u})_{u\in\Nbb}$ uniformly in $j\ge$1. Take $F(z)=\log z$, which is holomorphic on a neighborhood of $[s_{\min},s_{\max}]$, so for any $\gamma<1$,
\begin{equation}
\label{equation:sommabilite_c}
    \sup_{j\ge1}\sum_{u\in\Zbb} (1+|u|)^\gamma |c_{j,u}| <+\infty
\end{equation}
where 
\[
    c_{j,u} = \int_{0}^1 \log(s_j(\nu)) e^{2i\pi\nu u} \diff u.
\]
It is well known (see \cite[Theorem 17.17]{rudin1987real} and \cite{loubaton2020asymptotic}) that the sequence $c_{j,u}$ satisfies
\[
    h_j(\nu) = \exp\left(\frac{c_{j,0}}{2} + \sum_{u=1}^{+\infty} c_{j,u} e^{-2i\pi\nu u} \right).
\]
where we recall that $h(\nu)=\sum_{u\in\Nbb}a_{j,u}e^{-2i\pi u \nu}$ coincides with the outer spectral factor of $s_j(\nu)=|h_j(\nu)|^2$. We therefore see that the sequence of coefficients $(a_{j,u})_{u\in\Nbb}$ are related to $(c_{j,u})_{u\in\Nbb}$, and it can be shown (equation (D.11) in \cite{loubaton2020asymptotic}) that for each $\gamma<1$,
\[
    \sup_{j\ge 1}\sum_{u\ge0}(1+|u|)^\gamma|a_{j,u}|\le  \sup_{j\ge 1}\exp\left(\sum_{u\ge0}(1+|u|)^\gamma |c_{j,u}|\right)
\]
which by \eqref{equation:sommabilite_c} provides for any $\gamma<1$:
\begin{equation}
\label{equation:sommabilite_a}
    \sup_{j\ge 1}\sum_{u\ge0}(1+|u|)^\gamma|a_{j,u}|<+\infty.
\end{equation}
Returning to \eqref{equation:borne_R_1}, and using \eqref{equation:sommabilite_a}, we find
\[
    \sup_{j\ge1}\sum_{u=1}^{+\infty} |u|^{2\gamma}|a_{j,u}|^2 \le  \sup_{j\ge1}\left(\sum_{u\ge0} |u|^\gamma|a_{j,u}|\right)^2 < +\infty.
\]
Consider now the second term in \eqref{equation:borne_R_1}. For each $u\in\Nbb$, the quantity $u^{-2\gamma}\sup_{j_2\in[M],\nu\in\Fcal_N}|Z_{N,j_2,u}(\nu)|^2$ is positive so the monotone convergence theorem allows to exchange the sum and the expectation. 
\begin{multline*}
    \Ebb\left[\sum_{u=1}^{+\infty}\frac{1}{u^{2\gamma}}\max_{j_2\in[M], \nu\in\Fcal}|Z_{N,j_2,u}(\nu)|^2\right] = \sum_{u=1}^{+\infty}\frac{1}{u^{2\gamma}}\Ebb\left[\max_{j_2\in[M], \nu\in\Fcal}|Z_{N,j_2,u}(\nu)|^2\right]
\end{multline*}
so that equation \eqref{equation:borne_R_1} becomes
\begin{multline}
\label{equation:borne_R}
    \Ebb\left[\max_{j\in[M]}\max_{\nu\in\Fcal_N}R_{N,j}(\nu)\right]  \le \frac{C}{\sqrt{N}} \sqrt{\sum_{u=1}^{+\infty}\frac{1}{u^{2\gamma}}\Ebb\left[\max_{j_2\in[M], \nu\in\Fcal}|Z_{N,j_2,u}(\nu)|^2\right]}
\end{multline}
for some universal constant $C<+\infty$. To end the proof of Lemma \ref{lemma:controle_R}, it remains to show that for any $\eta>0$,
\[
    \sqrt{\sum_{u=1}^{+\infty}\frac{1}{u^{2\gamma}}\Ebb\left[\max_{j_2\in[M], \nu\in\Fcal_N}|Z_{N,j_2,u}(\nu)|^2\right]} = \Ocal(N^\eta)
\]
which is equivalent to show that for any $\eta>0$:
\begin{equation}
\label{equation:somme_esperance_sup_Z}
    \sum_{u=1}^{+\infty}\frac{1}{u^{2\gamma}}\Ebb\left[\max_{j_2\in[M], \nu\in\Fcal}|Z_{N,j_2,u}(\nu)|^2\right] = \Ocal(N^\eta).
\end{equation}
We now see that the behaviour of $\Ebb[\max_{j\in[M]}\max_{\nu\in\Fcal_N}R_{N,j}(\nu)]$ is governed by $\Ebb[\max_{j_2\in[M], \nu\in\Fcal_N}|Z_{N,j_2,u}(\nu)|^2]$, so it remains to study this quantity. 
By the triangle inequality:
\[
    \left|Z_{N,j,u}(\nu)\right| \le
    \left\{
    \begin{array}{ll}
        \left|\sum_{v=1-u}^{0}\epsilon_{j,v}e^{-2i\pi (v-1)\nu}\right| + \left|\sum_{v=N-u+1}^{N}\epsilon_{j,v}e^{-2i\pi (v-1)\nu}\right|  & \mbox{if } u< N \\
        \left|\sum_{v=1-u}^{N-u}\epsilon_{j,v}e^{-2i\pi (v-1)\nu}\right| + \left|\sum_{v=1}^{N}\epsilon_{j,v}e^{-2i\pi (v-1)\nu}\right|  & \mbox{if } u\ge N
    \end{array}
\right.
\]
and using the inequality $|a+b|^2\le 2(|a|^2+|b|^2)$:
\begin{multline}
\label{equation:borne_Z}
    \left|Z_{N,j,u}(\nu)\right|^2 \le \\
    \left\{
    \begin{array}{ll}
        2\left(\left|\sum_{v=1-u}^{0}\epsilon_{j,v}e^{-2i\pi (v-1)\nu}\right|^2 + \left|\sum_{v=N-u+1}^{N}\epsilon_{j,v}e^{-2i\pi (v-1)\nu}\right|^2\right)  & \mbox{if } u< N \\
        2\left(\left|\sum_{v=1-u}^{N-u}\epsilon_{j,v}e^{-2i\pi (v-1)\nu}\right|^2 + \left|\sum_{v=1}^{N}\epsilon_{j,v}e^{-2i\pi (v-1)\nu}\right|^2\right)  & \mbox{if } u\ge N.
    \end{array}
\right.
\end{multline}
In the case $u\ge N$, the two sums can be recognized as $N$ times the periodogram estimator which we defined previously as $I_{N,\epsilon_j}(\nu)$. Using the estimation \eqref{equation:controle_I_2} from Lemma \ref{lemma:controle_I}:
\begin{align}
\notag
    \Ebb\left[\max_{j\in[M]}\max_{\nu\in\Fcal}\left|\sum_{v=1-u}^{N-u}\epsilon_{j,v}e^{-2i\pi (v-1)\nu}\right|^2\right] &= N\Ebb\left[\max_{j\in[M]}\max_{\nu\in\Fcal}I_{N,\epsilon_j}(\nu)\right] \\
\label{equation:estimation_Z_somme_1}
    &= \Ocal\left(N \log N\right).
\end{align}
The other sum in the case $u\ge N$ is similar.

For $u<N$, the two sums have to be handled with more care for two reasons: the summation is only across $u$ terms (instead of $N$ terms) and the frequency $\nu$ is of the form $\frac{k}{N}$ instead of the required form $\frac{k}{u}$ to use the bound from Lemma \ref{lemma:controle_I} (said differently $\nu$ is no more a Fourier frequency for a sample size $u<N$). Therefore, we have to estimate the order of magnitude of $I_{u,\epsilon_j}(\nu)$ for $\nu\in[0,1]$ instead of $\nu\in\Fcal$. Lemma \ref{lemma:lemma_1_walker} and especially equation \eqref{equation:controle_I_continu} provides this. 
\begin{align}
\notag
    \Ebb\left[\max_{j\in[M]}\max_{\nu\in\Fcal}\left|\sum_{v=1-u}^{0}\epsilon_{j,v}e^{-2i\pi (v-1)\nu}\right|^2\right] &\le \Ebb\left[\max_{j\in[M]}\max_{\nu\in[0,1]}\left|\sum_{v=1-u}^{0}\epsilon_{j,v}e^{-2i\pi (v-1)\nu}\right|^2\right] \\
\notag
    &= \Ebb\left[\max_{j\in[M]}\max_{\nu\in[0,1]}uI_{u,\epsilon_j}(\nu)\right] \\
\label{equation:estimation_Z_somme_2}
    &= \Ocal\left(u\log u \left(\log u + \log M\right)\right).
\end{align}
The second sum in the case $u<N$ is also similar, therefore, collecting \eqref{equation:estimation_Z_somme_1} and \eqref{equation:estimation_Z_somme_2} in \eqref{equation:borne_Z}, we get:
\begin{equation}
\label{equation:controle_Z}
    \Ebb\left[\max_{j\in[M]}\max_{\nu\in\Fcal}\left|Z_{N,j,u}(\nu)\right|^2\right] =
    \left\{
    \begin{array}{ll}
        \Ocal\left(u\log u \log(uM)\right) & \mbox{if } u< N \\
        \Ocal\left(N\log N\right)  & \mbox{if } u\ge N.
    \end{array}
\right.
\end{equation}

It remains to use these bounds in the left hand side of \eqref{equation:somme_esperance_sup_Z}. 
\begin{multline*}
    \sum_{u=1}^{+\infty}\frac{1}{u^{2\gamma}}\Ebb\left[\max_{j_2\in[M], \nu\in\Fcal}|Z_{N,j_2,u}(\nu)|^2\right] = \sum_{u=1}^{N-1}\frac{1}{u^{2\gamma}}\Ebb\left[\max_{j_2\in[M], \nu\in\Fcal}|Z_{N,j_2,u}(\nu)|^2\right] \\ + \sum_{u=N}^{+\infty}\frac{1}{u^{2\gamma}}\Ebb\left[\max_{j_2\in[M], \nu\in\Fcal}|Z_{N,j_2,u}(\nu)|^2\right]
\end{multline*}

and using the estimates \eqref{equation:controle_Z}, 
\begin{align*}
    \sum_{u=1}^{+\infty}\frac{1}{u^{2\gamma}}\Ebb\left[\max_{j_2\in[M], \nu\in\Fcal}|Z_{N,j_2,u}(\nu)|^2\right] \le  \sum_{u=1}^{N-1}\frac{\log u \log(uM)}{u^{2\gamma-1}} +  \sum_{u=N}^{+\infty}\frac{N\log N}{u^{2\gamma}}.
\end{align*}
It is clear that 
\begin{align*}
    &\sum_{u=1}^{N-1}\frac{\log u \log(uM)}{u^{2\gamma-1}} = \Ocal\left(\frac{\log^2 N}{N^{2(\gamma-1)}}\right) \\
    &\sum_{u=N}^{+\infty}\frac{1}{u^{2\gamma}} = \Ocal\left(\frac{1}{N^{2\gamma-1}}\right)
\end{align*}
so for any $\gamma<1$,
\[
    \sum_{u=1}^{+\infty}\frac{1}{u^{2\gamma}}\Ebb\left[\max_{j_2\in[M], \nu\in\Fcal}|Z_{N,j_2,u}(\nu)|^2\right] = \Ocal\left(N^{2(1-\gamma)}\log^2 N + N^{2(1-\gamma)}\log N\right).
\]
This quantity is $\Ocal(N^{\eta})$ for any $\eta>0$, which proves \eqref{equation:somme_esperance_sup_Z} and ends the proof.
\end{proof}

Proposition \ref{proposition:controle_T} can now be proved.
\begin{proof}
Write $\hat{s}_{ij}(\nu) - \tilde{s}_{ij}(\nu)$ as:
\begin{align*}
    &\hat{s}_{ij}(\nu) - \tilde{s}_{ij}(\nu) =  \frac{1}{B+1}\sum_{b=-B/2}^{B/2}\xi_{y_i}\left(\nu+\frac{b}{N}\right)\overline{\xi_{y_j}\left(\nu+\frac{b}{N}\right)} \\ 
    &\quad- h_i\left(\nu+\frac{b}{N}\right)\xi_{\epsilon_i}\left(\nu+\frac{b}{N}\right)\overline{h_j\left(\nu+\frac{b}{N}\right)}\overline{\xi_{\epsilon_j}\left(\nu+\frac{b}{N}\right)} \\
    &= \frac{1}{B+1}\sum_{b=-B/2}^{B/2}\left( \xi_{y_i}\left(\nu+\frac{b}{N}\right) - h_i\left(\nu+\frac{b}{N}\right)\xi_{\epsilon_i}\left(\nu+\frac{b}{N}\right)\right)\overline{\xi_{y_j}\left(\nu+\frac{b}{N}\right)} \\
    &\quad+ h_i\left(\nu+\frac{b}{N}\right)\xi_{\epsilon_i}\left(\nu+\frac{b}{N}\right)\left(\overline{\xi_{y_j}\left(\nu+\frac{b}{N}\right)} - \overline{h_j\left(\nu+\frac{b}{N}\right)}\overline{\xi_{\epsilon_j}\left(\nu+\frac{b}{N}\right)}\right).
\end{align*}

We recognize the quantities $R_{N,i}(\nu)$ that have been bounded in Lemma \ref{lemma:controle_R}. It is now clear that:
\begin{multline}
\label{equation:decomposition_T}
    \max_{(i,j,\nu)\in\Ical}\left|\hat{s}_{ij}(\nu) - \tilde{s}_{ij}(\nu)\right| \le \max_{i\in[M],\nu\in\Fcal} R_{N,i}(\nu) \\ \times \left( \max_{i\in[M],\nu\in\Fcal}\left|\xi_{y_i}(\nu)\right| +  \max_{i\in[M],\nu\in\Fcal}\left|h_i(\nu)\xi_{\epsilon_i}(\nu)\right| \right).
\end{multline}
By Lemma \ref{lemma:controle_I}:
\[
    \max_{i\in[M], \nu\in\Fcal}|\xi_{\epsilon_i}(\nu)| = \Ocal_P\left(\sqrt{\log M + \log N}\right) = \Ocal_P\left(\sqrt{\log N}\right)
\]
and in conjunction with Lemma \ref{lemma:controle_R}, for any $\delta<1/2$,
\[
    \max_{i\in[M]}\max_{\nu\in\Fcal}|\xi_{y_i}(\nu)| \le \underbrace{\max_{i\in[M]}\max_{\nu\in\Fcal}|h_i(\nu)\xi_{\epsilon_i}(\nu)|}_{\Ocal_P(\sqrt{\log N})} + \underbrace{\max_{i\in[M]}\max_{\nu\in\Fcal}|R_{N,i}(\nu)|}_{\Ocal_P(N^{-\delta})}
\]
which is $\Ocal_P\left(\sqrt{\log N}\right)$. Each quantity involved in \eqref{equation:decomposition_T} is now estimated, and provides, for any $\delta<1/2$:
\[
    \max_{(i,j,\nu)\in\Ical}\left|\hat{s}_{ij}(\nu) - \tilde{s}_{ij}(\nu)\right| = \Ocal_P(N^{-\delta}\sqrt{\log N}) = \Ocal_P(N^{-\delta'})
\]
for any $\delta'<1/2$. By Assumption \ref{assumption:rate_NBM}, $\rho<1$, ie. $\sqrt{B+1}=o(N^{1/2})$ therefore one can always take $\delta'=\frac{\rho/2 + 1/2}{2}\in(0,1/2)$ such that: 
\[
    \sqrt{B+1}\max_{(i,j,\nu)\in\Ical}\left|\hat{s}_{ij}(\nu) - \tilde{s}_{ij}(\nu)\right| = \Ocal_P(N^{-\delta'})
\]
and we get \eqref{equation:controle_T}.
\end{proof}

\section{Proof of Proposition \ref{proposition:deviation_modere}: moderate deviations of $\tilde{s}_{ij}(\nu)$}
\label{appendix:moderate_deviations}


First, we give two preliminary lemmas regarding the concentration of Gaussian sesquilinear forms.
\begin{lemma}
\label{lemma:sesquilinear_expectation}
Let $\x,\y$ independent $\Ncal_{\Cbb^M}(\bf{0},\I)$ random vectors and $\A$ a non-zero $M\times M$ deterministic matrix. For any $t>0$,
\begin{equation}
\label{equation:s_tail_expectation_relationship}
    \Prob\left(|\x^*\A\y|>t\right) = \Ebb\left[\exp\left(-\frac{t^2}{\y^*\A^*\A\y}\right) \right].
\end{equation}
Moreover, if $\z\sim\Ncal_{\Cbb^M}(\bf{0},\I)$ is jointly independent from $\x$ and $\y$, and $\B$ is another non zero $M\times M$ deterministic matrix, for any $t,s>0$,
\begin{equation}
\label{equation:s_tail_joint_expectation_relationship}
    \Prob\left(|\x^*\A\y|>t,\ |\z^*\B\y|>s \right) = \Ebb\left[\exp\left(-\frac{t^2}{\y^*\A^*\A\y}-\frac{s^2}{\y^*\B^*\B\y}\right) \right].
\end{equation}
\end{lemma}
The proof of Lemma \ref{lemma:sesquilinear_expectation} is straightforward and therefore omitted.

The next lemma is the Hanson-Wright inequality \cite{rudelson2013hanson} in the special case of a sesquilinear form.
\begin{lemma}
Let $\x,\y$ be independent $\Ncal_{\Cbb^M}(\bf{0},\I)$ random variables, and $\A$ a deterministic $M\times M$ matrix. Then, for any $t\ge0$:
\[
    \Prob\left(\left|\x^*\A\y - \Ebb\left[\x^*\A\y\right]\right|>t\right) \le 2\exp\left(-C\min\left(\frac{t}{\|\A\|}, \frac{t^2}{\|\A\|_F^2}\right)\right)
\]
where $C$ is a universal constant (independent of $t$ and $\A$).
\end{lemma}


In order to prove Proposition \ref{proposition:deviation_modere}, we recall from \eqref{equation:s_tilde_sesquilinear} that $\tilde{s}_{ij}(\nu)$ may be written as the Gaussian sesquilinear form
\begin{align*}
  \tilde{s}_{ij}(\nu) =  \xibs_{\epsilon_j}(\nu)^* \frac{\Pibs_{ij}(\nu)}{\sqrt{B+1}}\xibs_{\epsilon_i}(\nu)
\end{align*}
where $\xibs_{\epsilon_1}(\nu),\ldots,\xibs_{\epsilon_M}(\nu)$ are i.i.d. $\Ncal_{\Cbb^{B+1}}(\mathbf{0},\I)$ distributed, and that we denote $\sigma^2_{ij}(\nu) = \frac{1}{B+1} \Tr \Sigmabs_{ij}(\nu)$ with
\begin{align*}
  \Sigmabs_{ij}(\nu)
  &= \Pibs_{ij}(\nu)\Pibs_{ij}(\nu)^*
    \\
  &= \diag\left(s_i\left(\nu + \frac{b}{N}\right)s_j\left(\nu + \frac{b}{N}\right): b = -\frac{B}{2},\ldots,\frac{B}{2}\right).
    \notag
\end{align*}
Note also that thanks to Assumptions \ref{assumption:short_memory} and \ref{assumption:non_vanishing_spectrum}, there exist $s_{\min}, s_{\max} > 0$ such that:
\[
    0 < s_{\min} \le \inf_{m\ge1}\min_{\nu\in[0,1]}s_m(\nu) \le \sup_{m\ge1}\max_{\nu\in[0,1]}s_m(\nu) \le s_{\max} < +\infty
  \]
  and consequently, the following inequality holds:
\begin{align}
  0
  <
  s_{\min}^2
  \le
  \inf_{N \geq 1}\min_{(i,j,\nu)\in\Ical} \lambda_{\mathrm{min}}\left(\Sigmabs_{ij}(\nu)\right)
  <
  \sup_{N \geq 1} \max_{(i,j,\nu)\in\Ical} \lambda_{\mathrm{max}}\left(\Sigmabs_{ij}(\nu)\right)
  \le
  s_{\max}^2
  < +\infty,
  \label{eq:borne_Sigmaij}
\end{align}
where $\lambda_{\mathrm{min}}\left(\Sigmabs_{ij}(\nu)\right)$, $\lambda_{\mathrm{max}}\left(\Sigmabs_{ij}(\nu)\right)$ are respectively the smallest and largest eigenvalue (or diagonal entry) of $\Sigmabs_{ij}(\nu)$.

In the remainder, to lighten the presentation, we use the multi-index $\alpha$ instead of $(i,j,\nu)$ as well as the notation $\cdot_\alpha$ in place of $\cdot_{ij}(\nu)$ so that, for example, $\tilde{s}_{ij}(\nu)$, $\Pibs_{ij}(\nu)$, $\Sigmabs_{ij}(\nu)$ become $\tilde{s}_{\alpha}$, $\Pibs_{\alpha}$ and $\Sigmabs_{\alpha}$ respectively.


From Lemma \ref{lemma:sesquilinear_expectation}, the probabilities appearing in  \eqref{equation:moderate_deviation_sij} and \eqref{equation:joint_deviation_modere} in the statement of Proposition \ref{proposition:deviation_modere} can be rewritten as
\begin{align}
\label{equation:proba_esperance}
    \Prob\left((B+1)\frac{|\tilde{s}_\alpha|^2}{\sigma_\alpha^2}>t^2\right) = \Ebb\left[\exp\left(-\frac{\Tr \Sigmabs_\alpha}{\w^*\Sigmabs_\alpha\w}t^2\right) \right]
\end{align}
and
\begin{multline}
  \Prob\left((B+1)\frac{|\tilde{s}_\alpha|^2}{\sigma_\alpha^2}>t^2, (B+1)\frac{|\tilde{s}_{\alpha'}|^2}{\sigma_{\alpha'}^2}>s^2\right)
  \\
  = \Ebb\left[\exp\left(-\left(\frac{\Tr \Sigmabs_{\alpha}}{\w^*\Sigmabs_{\alpha}\w}t^2 + \frac{\Tr \Sigmabs_{\alpha'}}{\w^*\Sigmabs_{\alpha'}\w}s^2\right)\right)\right].
  \label{eq:proba_esperance_joint}
\end{multline}
for some $\w\sim\Ncal_{\Cbb^{B+1}}(\bf{0},\I)$.
The next two lemmas are dedicated to the study of the concentration of $\Tr\Sigmabs_\alpha/\w^*\Sigmabs_\alpha\w$ around 1.
\begin{lemma}
  \label{lemma:conc_Sigmaalpha}
  There exists two universal constants $C_1,C_2$ such that for all $t \in (0,1)$,
  \begin{align}
    \max_{\alpha\in\Ical}\Prob\left(\left|\frac{\Tr\Sigmabs_\alpha}{\w^*\Sigmabs_\alpha\w} -1 \right|> t\right) \le C_1\exp\left(-C_2 B t^2\right).
    \label{eq:conc_Sigmaalpha}
  \end{align}
\end{lemma}
\begin{proof}
  We have
  \begin{align}
    &\max_{\alpha \in \Ical}
    \Prob\left(\left|\frac{\Tr\Sigmabs_\alpha}{\w^*\Sigmabs_\alpha\w} -1 \right|> t\right)
    \notag\\  
    &\qquad\qquad
      = \max_{\alpha \in \Ical} \Prob\left(\w^*\Sigmabs_\alpha\w - \Tr\Sigmabs_\alpha \in\left[\frac{-t}{1+t} \Tr\Sigmabs_\alpha,\frac{t}{1-t} \Tr\Sigmabs_\alpha\right]^c\right)
      \notag\\
    &\qquad\qquad
      \le \max_{\alpha \in \Ical}\Prob\left(\w^*\Sigmabs_\alpha\w - \Tr\Sigmabs_\alpha  \in\left[-\frac{t}{2} \Tr\Sigmabs_\alpha,\frac{t}{2} \Tr\Sigmabs_\alpha\right]^c\right)
    \notag\\
    &\qquad\qquad
      \le \max_{\alpha \in \Ical} \Prob\left(\left|\w^*\Sigmabs_\alpha\w - \Tr\Sigmabs_\alpha\right| > \frac{t}{2} \Tr\Sigmabs_\alpha\right).
      \notag
  \end{align}  
  Since $\w \sim \Ncal_{\Cbb^{B+1}}(\bf{0},\I)$, the Hanson-Wright inequality \cite{rudelson2013hanson} provides that
\begin{align*}
  &\Prob\left(\left|\w^*\Sigmabs_\alpha\w - \Tr\Sigmabs_\alpha\right| > \frac{t}{2} \Tr\Sigmabs_\alpha\right)
    \notag\\
  &\qquad\qquad\leq 2\exp\left(-C\min\left(\frac{\Tr\Sigmabs_\alpha}{\|\Sigmabs_\alpha\|} t, \frac{(\Tr\Sigmabs_\alpha)^2}{\|\Sigmabs_\alpha\|_F^2} t^2 \right) \right)
  \\
  &\qquad\qquad\le 2\exp\left(-C t^2 \min\left(\frac{\Tr\Sigmabs_\alpha}{\|\Sigmabs_\alpha\|}, \frac{(\Tr\Sigmabs_\alpha)^2}{\|\Sigmabs_\alpha\|_F^2} \right) \right)
\end{align*}
for some universal constant $C$, where $\|\Sigmabs_\alpha\|$ and $\|\Sigmabs_\alpha\|_F$ denote the spectral norm and Frobenius norm of $\Sigmabs_{\alpha}$ respectively.
From \eqref{eq:borne_Sigmaij}, we also have
\begin{align*}
    &\min_{\alpha\in\Ical}\frac{\Tr\Sigmabs_\alpha}{\|\Sigmabs_\alpha\|}\ge  (B+1)\frac{s_{\min}^2}{s_{\max}^2} \\
    &\min_{\alpha\in\Ical}\frac{(\Tr\Sigmabs_\alpha)^2}{\|\Sigmabs_\alpha\|_F^2}  (B+1)\frac{s_{\min}^4}{s_{\max}^4}.
\end{align*}
Consequently, we can find some universal constants $C_1,C_2$ such that for all $N \geq 1$, \eqref{eq:conc_Sigmaalpha} holds.
\end{proof}

\begin{lemma}
  \label{lemma:esperance_fq_tr}
  For any $\beta \in \left(0,\frac{1}{2}\right)$,
\begin{equation}
\label{equation:controle_esperance_fq_tr}
    \max_{\alpha\in\Ical} \left| \Ebb\left[ \frac{\Tr\Sigmabs_\alpha}{\w^*\Sigmabs_\alpha\w} \right] - 1 \right| = \Ocal(B^{-\beta})
\end{equation}
\end{lemma}
\begin{proof}
Define the event 
\[
    \Omega_{\alpha,N}:= \left\{ \left|\frac{\Tr\Sigmabs_\alpha}{\w^*\Sigmabs_\alpha\w} -1 \right| < \kappa_N \right\}
\]
where $\kappa_N$ is some sequence satisfying $\kappa_N\to 0$ as $N\to+\infty$ and consider the decomposition
\begin{equation}
\label{equation:decomposition_esperance}
\left|\Ebb\left[ \frac{\Tr\Sigmabs_\alpha}{\w^*\Sigmabs_\alpha\w} \right] -1\right|
\leq
\left|\Ebb\left[ \frac{\Tr\Sigmabs_\alpha}{\w^*\Sigmabs_\alpha\w}  \mathds{1}_{\Omega_{\alpha,N}}\right] - 1 \right|
+ \Ebb\left[ \frac{\Tr\Sigmabs_\alpha}{\w^*\Sigmabs_\alpha\w}  \mathds{1}_{\Omega_{\alpha,N}^c}\right].
\end{equation}
For the first term of the right-hand side of \eqref{equation:decomposition_esperance}, the following bound holds:
\begin{equation}
\label{equation:gestion_Ebb_1}
    \left|\Ebb\left[\frac{\Tr\Sigmabs_\alpha}{\w^*\Sigmabs_\alpha\w}\mathds{1}_{\Omega_{\alpha,N}}\right] -1\right| \leq \kappa_N + \Pbb\left(\Omega_{\alpha,N}^c\right).
\end{equation}
Regarding the second term, Cauchy-Schwarz inequality implies that
\begin{equation}
\label{equation:decomposition_esperance_cauchy_schwartz}
    \Ebb\left[ \frac{\Tr\Sigmabs_\alpha}{\w^*\Sigmabs_\alpha\w}  \mathds{1}_{\Omega_{\alpha,N}^c}\right] \le \sqrt{\Ebb\left[ \left|\frac{\Tr\Sigmabs_\alpha}{\w^*\Sigmabs_\alpha\w}\right|^2\right]} \sqrt{\Prob\left(\Omega_{\alpha,N}^c\right)}.
\end{equation}


Using \eqref{eq:borne_Sigmaij}, we have
\begin{equation*} 
    \max_{\alpha\in\Ical}\frac{\Tr\Sigmabs_\alpha}{\w^*\Sigmabs_\alpha\w} \le (B+1)\frac{s_{\max}^2}{s_{\min}^2} \frac{1}{\|\w\|^2}.
  \end{equation*}
  and since $\frac{1}{2 \|\w\|^2}$ is distributed as an inverse--$\chi^2$ random variable with $2(B+1)$ degrees of freedom, we have from \cite[Appendix A6]{Robert2007} that
\[
    \Ebb\left[\frac{1}{\|\w\|^4}\right]=\Ocal\left(\frac{1}{B^2}\right)
  \]
  which yields to 
\begin{equation*}
  \sup_{N \geq 1}\max_{\alpha\in\Ical}\Ebb\left[ \left|\frac{\Tr\Sigmabs_\alpha}{\w^*\Sigmabs_\alpha\w}\right|^2\right] < +\infty.
  \notag
\end{equation*}
Consequently, gathering \eqref{equation:gestion_Ebb_1} and \eqref{equation:decomposition_esperance_cauchy_schwartz} and using Lemma \ref{lemma:conc_Sigmaalpha}, we get
\begin{align*}
  \max_{\alpha \in \Ical}\left|\Ebb\left[ \frac{\Tr\Sigmabs_\alpha}{\w^*\Sigmabs_\alpha\w} \right] -1\right|
  &\leq \kappa_N + C \sqrt{\Pbb\left(\Omega_{\alpha,N}^c\right)}
  \\
  & \leq \kappa_N + C_1\exp\{-C_2\kappa_N^2 B\}
\end{align*}
for some universal constants $C_1,C_2$. Choosing $\kappa_N = B^{-\beta}$ with $\beta \in (0,1/2)$ yields the desired result.

\end{proof}

Before proving Proposition \ref{proposition:deviation_modere}, we need one last result on the concentration of $\frac{\Tr\Sigmabs_\alpha}{\w^*\Sigmabs_\alpha\w}$ around its mean, which is a straightforward consequence of previous Lemmas \ref{lemma:conc_Sigmaalpha} and \ref{lemma:esperance_fq_tr}.
\begin{lemma}
  \label{lemma:hanson_wright_tr_fq}
  Let $\delta\in(0,\frac{1}{2})$ and $(\epsilon_N)_{N \geq 1}$ some non-negative sequence converging towards $0$ as $N\to\infty$ and such that $\epsilon_N B^\delta\to+\infty$.
  Then, there exist two universal constants $C_1,C_2$ such that
\begin{align*}
  \max_{\alpha\in\Ical}\Prob\left(\left|\frac{\Tr\Sigmabs_\alpha}{\w^*\Sigmabs_\alpha\w} - \Ebb\left[\frac{\Tr\Sigmabs_\alpha}{\w^*\Sigmabs_\alpha\w}\right] \right|>\epsilon_N\right)
  \le C_1\exp\left(-C_2\epsilon_N^2 B\right).
\end{align*}
\end{lemma}

\begin{proof}
Write:
\begin{multline*}
  \Prob\left(\left|\frac{\Tr\Sigmabs_\alpha}{\w^*\Sigmabs_\alpha\w} - \Ebb\left[\frac{\Tr\Sigmabs_\alpha}{\w^*\Sigmabs_\alpha\w}\right] \right|>\epsilon_N\right)
  \\
  \le
  \Prob
  \left(
    \left|\frac{\Tr\Sigmabs_\alpha}{\w^*\Sigmabs_\alpha\w} - 1\right|
    >
    \epsilon_N
    - \left|
      1-\Ebb\left[\frac{\Tr\Sigmabs_\alpha}{\w^*\Sigmabs_\alpha\w}\right]
    \right|
  \right).
  \end{multline*}
From Lemma \ref{lemma:esperance_fq_tr}, there exists a universal constant $C$ such that
\begin{align*}
  \max_{\alpha \in \Ical}\left|1-\Ebb\left[\frac{\Tr\Sigmabs_\alpha}{\w^*\Sigmabs_\alpha\w}\right]\right|  \leq \frac{C}{B^{\delta}}.
  \notag
\end{align*}
Moreover, by assumption on the rate of $\epsilon_N$, we have $\frac{C}{B^{\delta}} < \frac{\epsilon_N}{2}$ for all large $N$.
Consequently,
\begin{align}
  \max_{\alpha \in \Ical}\Prob\left(\left|\frac{\Tr\Sigmabs_\alpha}{\w^*\Sigmabs_\alpha\w} - \Ebb\left[\frac{\Tr\Sigmabs_\alpha}{\w^*\Sigmabs_\alpha\w}\right] \right|>\epsilon_N\right)
  \leq
  \max_{\alpha \in \Ical}\Prob
  \left(
    \left|\frac{\Tr\Sigmabs_\alpha}{\w^*\Sigmabs_\alpha\w} - 1\right|
    >
    \frac{\epsilon_N}{2}
  \right)
  \label{eq:tmpsigmaalpha}
\end{align}
for all large $N$. Applying directly Lemma \ref{lemma:conc_Sigmaalpha} to \eqref{eq:tmpsigmaalpha} allows to conclude the proof.

\end{proof}

Endowed with Lemmas \ref{lemma:conc_Sigmaalpha}, \ref{lemma:esperance_fq_tr} and \ref{lemma:hanson_wright_tr_fq},
we are now in position to complete the proof of Proposition \ref{proposition:deviation_modere}.

We first tackle \eqref{equation:moderate_deviation_sij} and show as a first step that there exists $\eta > 0$ such that for any universal constant $C$, 
\begin{align}
  \max_{t\in[0,C B^{\eta}]}\max_{\alpha\in\Ical}
  \left|\Prob\left((B+1)\frac{|\tilde{s}_\alpha|^2}{\sigma_\alpha^2}>t^2\right) \exp\left(\Ebb\left[\frac{\Tr\Sigmabs_\alpha}{\w^*\Sigmabs_\alpha\w}t^2\right]\right) -  1\right| = o(1).
  \label{eq:proba_salpha_tmp}
\end{align}
Let $\delta \in \left(0,\frac{1}{2}\right)$ and $(\epsilon_N)_{N \geq 1}$ some non-negative sequence converging to $0$ and satisfying $\epsilon_N B^\delta\to+\infty$, and define the event
\begin{equation}
\label{equation:definition_Theta_ijnu}
    \Theta_{\alpha,N} := \left\{ \left|\frac{\Tr\Sigmabs_\alpha}{\w^*\Sigmabs_\alpha\w} - \Ebb\left[\frac{\Tr\Sigmabs_\alpha}{\w^*\Sigmabs_\alpha\w}\right] \right| < \epsilon_N \right\}
  \end{equation}
as in Lemma \ref{lemma:hanson_wright_tr_fq}.
Next, consider the decomposition
\begin{align}
\notag
    &\Ebb\left[ \exp\left(-\left(\frac{\Tr\Sigmabs_\alpha}{\w^*\Sigmabs_\alpha\w} - \Ebb\left[\frac{\Tr\Sigmabs_\alpha}{\w^*\Sigmabs_\alpha\w}\right] \right)t^2\right) \right] \\
\notag
    &= \Ebb\left[ \exp\left(-\left(\frac{\Tr\Sigmabs_\alpha}{\w^*\Sigmabs_\alpha\w} - \Ebb\left[\frac{\Tr\Sigmabs_\alpha}{\w^*\Sigmabs_\alpha\w}\right] \right)t^2\right) \mathds{1}_{\Theta_{\alpha,N}} \right] \\
\notag
    &\quad+\Ebb\left[ \exp\left(-\left(\frac{\Tr\Sigmabs_\alpha}{\w^*\Sigmabs_\alpha\w} - \Ebb\left[\frac{\Tr\Sigmabs_\alpha}{\w^*\Sigmabs_\alpha\w}\right] \right)t^2\right) \mathds{1}_{\Theta_{\alpha,N}^c} \right] \\
\label{equation:decomposition_esperance_exp}
    &:= \Psi_{\alpha,N}(t) + \Delta_{\alpha,N}(t).
\end{align}

On the event $\Theta_{\alpha,N}$, we have
\begin{equation*}
    \exp\left(-\left(\frac{\Tr\Sigmabs_\alpha}{\w^*\Sigmabs_\alpha\w} - \Ebb\left[\frac{\Tr\Sigmabs_\alpha}{\w^*\Sigmabs_\alpha\w}\right] \right)t^2\right) \in \left[\exp\left(-\epsilon_N t^2\right), \exp \left(\epsilon_N t^2\right)\right]
\end{equation*}
which implies, that:
\begin{align}
  \left|\Psi_{\alpha,N}(t)-1\right|
  &\le \max\left(1-\Prob[\Theta_{\alpha,N}]e^{-\epsilon_Nt^2}, \Prob[\Theta_{\alpha,N}]e^{\epsilon_Nt^2}-1\right)
  \notag\\
  &\le \left|e^{\epsilon_N t^2}-1\right| + \left(1-e^{-\epsilon_N t^2}\right) + (e^{\epsilon_N t^2}+e^{-\epsilon_N t^2})\Prob[\Theta_{\alpha,N}^c].
  \notag
\end{align}
Using Lemma \ref{lemma:hanson_wright_tr_fq}, we further have
\begin{multline}
\label{equation:controle_Psi}
\max_{\alpha\in\Ical}\left|\Psi_{\alpha,N}(t)-1\right|
\\\le \left|e^{\epsilon_N t^2}-1\right| + \left(1-e^{-\epsilon_N t^2}\right) + (e^{\epsilon_N t^2}+e^{-\epsilon_N t^2}) C_1e^{-C_2\epsilon_N^2 B}
\end{multline}
for some universal constants $C_1,C_2$.
Regarding $\Delta_{\alpha,N}(t)$, we clearly have
\begin{align}
  \left|\Delta_{\alpha,N}(t)\right|
  \leq
  \Prob\left(\Theta_{\alpha,N}^c\right) \exp\left(\Ebb\left[\frac{\Tr\Sigmabs_\alpha}{\w^*\Sigmabs_\alpha\w}\right] t^2\right).
  \notag
\end{align}
Using Lemmas \ref{lemma:esperance_fq_tr} and \ref{lemma:hanson_wright_tr_fq}, for any $\beta \in \left(0,\frac{1}{2}\right)$, there exists a universal constant $C_3$ such that
\begin{equation} 
  \max_{\alpha\in\Ical}\left|\Delta_{\alpha,N}(t)\right| \le C_1 \exp\left(-C_2 B\epsilon_N^2 + \left(1+\frac{C_3}{B^{\beta}}\right) t^2\right).
  \label{equation:controle_Delta_final}
\end{equation}
Combining \eqref{equation:decomposition_esperance_exp}, \eqref{equation:controle_Psi} and  \eqref{equation:controle_Delta_final}, one gets
\begin{align*}
  &\max_{\alpha\in\Ical}
    \left|\Ebb\left[ \exp\left(-\left(\frac{\Tr\Sigmabs_\alpha}{\w^*\Sigmabs_\alpha\w} - \Ebb\left[\frac{\Tr\Sigmabs_\alpha}{\w^*\Sigmabs_\alpha\w}\right] \right)t^2\right) \right] -1\right|
  \\
  &\qquad\qquad\qquad\le
    \left|e^{\epsilon_N t^2}-1\right| + \left(1-e^{-\epsilon_N t^2}\right) + (e^{\epsilon_N t^2}+e^{-\epsilon_N t^2}) C_1e^{-C_2\epsilon_N^2 B}
    \\
   &\qquad\qquad\qquad\qquad + C_1 \exp\left(-C_2 B\epsilon_N^2 + \left(1+\frac{C_3}{B^{\beta}}\right) t^2\right).
\end{align*}
Set $\epsilon_N = B^{-\frac{\delta}{2}}$ so that $\epsilon_N \to 0$ and $\epsilon_N B^{\delta} \to +\infty$ as required, and let $\eta = \frac{\delta}{8}$.
Then, recalling that $\delta \in \left(0,\frac{1}{2}\right)$, we have
\begin{align*}
  &\epsilon_N B^{2 \eta} = \frac{1}{B^{\frac{\delta}{4}}} \xrightarrow[N\to+\infty]{} 0, \\
  &\epsilon_N^2 B = B^{1-\delta} \xrightarrow[N\to+\infty]{} +\infty, \\
      & \frac{B^{2\eta}}{\epsilon_N^2 B} = B^{2 \eta + \delta -1} = B^{\frac{5}{4}\delta -1} \xrightarrow[N\to+\infty]{} 0.
        \notag
\end{align*}
Therefore, for any universal constant $C$,
\begin{align*} 
  \max_{t \in [0, CB^{\eta}]}\max_{\alpha\in\Ical}
  \left|\Ebb\left[ \exp\left(-\left(\frac{\Tr\Sigmabs_\alpha}{\w^*\Sigmabs_\alpha\w} - \Ebb\left[\frac{\Tr\Sigmabs_\alpha}{\w^*\Sigmabs_\alpha\w}\right] \right)t^2\right) \right] -1\right|
  \xrightarrow[N\to\infty]{} 0,
\end{align*}
which, thanks to \eqref{equation:proba_esperance}, implies \eqref{eq:proba_salpha_tmp}.
Finally, using Lemma \ref{lemma:sesquilinear_expectation}, we deduce that
\begin{align}
  &\max_{t \in [0,CB^{\eta}]} \max_{\alpha \in \Ical}
  \Prob\left((B+1)\frac{|\tilde{s}_\alpha|^2}{\sigma_\alpha^2}>t^2\right)
  \left|
  \exp\left(\Ebb\left[\frac{\Tr\Sigmabs_\alpha}{\w^*\Sigmabs_\alpha\w}\right] t^2\right) - e^{t^2}
  \right|
   \label{eq:proba_salpha_diff}
  \\
  &\qquad\qquad\leq
    \max_{t \in [0,CB^{\eta}]} \max_{\alpha \in \Ical}
    \Prob\left((B+1)\frac{|\tilde{s}_\alpha|^2}{\sigma_\alpha^2}>t^2\right)  \exp\left(\Ebb\left[\frac{\Tr\Sigmabs_\alpha}{\w^*\Sigmabs_\alpha\w}\right] t^2\right)
    \notag\\
  &\qquad\qquad\qquad\qquad\times
    \max_{t \in [0,CB^{\eta}]} \max_{\alpha \in \Ical}
    \left|1 - \exp\left(\left(1-\Ebb\left[\frac{\Tr\Sigmabs_\alpha}{\w^*\Sigmabs_\alpha\w}\right]\right) t^2\right)\right|
  \notag\\
  &\qquad\qquad\leq \left(1+o(1)\right) \left(1 - \exp\left(\Ocal\left(\frac{B^{2\eta}}{B^{\delta}}\right)\right)\right)
    \notag\\
  &\qquad\qquad \xrightarrow[N\to\infty]{} 0,
    \notag
\end{align}
which, combined with \eqref{equation:proba_esperance}, shows \eqref{equation:moderate_deviation_sij}.
We now turn to \eqref{equation:joint_deviation_modere}. Since the proof is very similar to the one of \eqref{equation:moderate_deviation_sij}, we only provide the main steps.
Using \eqref{eq:proba_esperance_joint}, we consider the following decomposition:
\begin{multline*}
  \Prob\left((B+1)\frac{|\tilde{s}_{\alpha}|^2}{\sigma_\alpha^2}>t, (B+1)\frac{|\tilde{s}_{\alpha'}|^2}{\sigma_{\alpha'}^2}>s\right)
  \\
  \times\exp\left(\Ebb\left[\frac{\Tr \Sigmabs_{\alpha}}{\w^*\Sigmabs_{\alpha}\w}\right]t^2 + \Ebb\left[\frac{\Tr \Sigmabs_{\alpha'}}{\w^*\Sigmabs_{\alpha'}\w}\right]s^2\right)
  \\
  := \Psi_{\alpha,\alpha',N}(t,s) + \Delta_{\alpha,\alpha',N}(t,s)
\end{multline*}
where 
\begin{align*}
  &\Psi_{\alpha,\alpha',N}(t,s) =
  \\
  &\qquad\qquad\Ebb\Biggl[\exp\left(-\left(\frac{\Tr \Sigmabs_{\alpha}}{\w^*\Sigmabs_{\alpha}\w} - \Ebb\left[\frac{\Tr \Sigmabs_{\alpha}}{\w^*\Sigmabs_{\alpha}\w}\right]\right)t^2\right)
  \\
  &\qquad\qquad\qquad\times \exp\left(-\left(\frac{\Tr \Sigmabs_{\alpha'}}{\w^*\Sigmabs_{\alpha'}\w} - \Ebb\left[\frac{\Tr \Sigmabs_{\alpha'}}{\w^*\Sigmabs_{\alpha'}\w}\right]\right)s^2\right)
  \mathds{1}(\Theta_{\alpha,N}\cap\Theta_{\alpha',N}) \Biggr]
\end{align*}
and
\begin{align*}
  &\Delta_{\alpha,\alpha',N}(t,s) =
  \\
  &\qquad\qquad
    \Ebb\Biggl[\exp\left(-\left(\frac{\Tr \Sigmabs_{\alpha}}{\w^*\Sigmabs_{\alpha}\w} - \Ebb\left[\frac{\Tr \Sigmabs_{\alpha}}{\w^*\Sigmabs_{\alpha}\w}\right]\right)t^2\right)
  \\ 
  &\qquad\qquad\qquad
    \times \exp\left(-\left(\frac{\Tr \Sigmabs_{\alpha'}}{\w^*\Sigmabs_{\alpha'}\w} - \Ebb\left[\frac{\Tr \Sigmabs_{\alpha'}}{\w^*\Sigmabs_{\alpha'}\w}\right]\right)s^2\right) \mathds{1}(\Theta_{\alpha,N}^c\cup\Theta_{\alpha',N}^c) \Biggr].
\end{align*}
Using exactly the same arguments as for \eqref{equation:controle_Psi} and \eqref{equation:controle_Delta_final} and keeping the same requirements as above regarding the behaviour of sequence $(\epsilon_N)_{N \geq 1}$ and constant $\eta$, we may show that
\begin{equation*}
  \max_{t,s\in\left[0,C B^{\eta}\right]}\max_{\substack{\alpha\in\Ical\\\alpha'\in\Ical_\alpha}}|\Psi_{\alpha,\alpha',N}(t,s)-1| \xrightarrow[N\to\infty]{} 0,
\end{equation*}
as well as 
\begin{equation*}
        \max_{t,s\in \left[0,C B^{\eta}\right]}\max_{\substack{\alpha\in\Ical\\\alpha'\in\Ical_\alpha}}\Delta_{\alpha,\alpha',N}(t,s) \xrightarrow[N\to\infty]{} 0.
 \end{equation*}
 Consequently,
 \begin{align*}
   &\max_{t,s\in \left[0,C B^{\eta}\right]}\max_{\substack{\alpha\in\Ical\\\alpha'\in\Ical_\alpha}}
   \Biggl|\Prob\left((B+1)\frac{|\tilde{s}_{\alpha}|^2}{\sigma_\alpha^2}>t, (B+1)\frac{|\tilde{s}_{\alpha'}|^2}{\sigma_{\alpha'}^2}>s\right)
   \\
   &\qquad\qquad\qquad\qquad\qquad\qquad
     \times\exp\left(\Ebb\left[\frac{\Tr \Sigmabs_{\alpha}}{\w^*\Sigmabs_{\alpha}\w}\right]t^2 + \Ebb\left[\frac{\Tr \Sigmabs_{\alpha'}}{\w^*\Sigmabs_{\alpha'}\w}\right]s^2\right)
     - 1
     \Biggr|
     \notag\\
   &\qquad\qquad\qquad\qquad\qquad\qquad
     \xrightarrow[N\to\infty]{} 0.
     \notag
\end{align*}
As for \eqref{eq:proba_salpha_diff}, we also have using a similar bound,
\begin{align*}
   &\max_{t,s\in \left[0,C B^{\eta}\right]}\max_{\substack{\alpha\in\Ical\\\alpha'\in\Ical_\alpha}}
   \Prob\left((B+1)\frac{|\tilde{s}_{\alpha}|^2}{\sigma_\alpha^2}>t, (B+1)\frac{|\tilde{s}_{\alpha'}|^2}{\sigma_{\alpha'}^2}>s\right)
   \\
   &\qquad\qquad\qquad\qquad
     \times\Biggl|\exp\left(\Ebb\left[\frac{\Tr \Sigmabs_{\alpha}}{\w^*\Sigmabs_{\alpha}\w}\right]t^2 + \Ebb\left[\frac{\Tr \Sigmabs_{\alpha'}}{\w^*\Sigmabs_{\alpha'}\w}\right]s^2\right)
     - \exp\left(t^2+s^2\right)
     \Biggr|
     \notag\\
   &\qquad\qquad\qquad\qquad\qquad\qquad
     \xrightarrow[N\to\infty]{} 0.
     \notag
\end{align*}
The two previous convergences combined together complete the proof of \eqref{equation:joint_deviation_modere}.

\bibliographystyle{elsarticle-harv} 
\bibliography{references}

\begin{thebibliography}{27}
\expandafter\ifx\csname natexlab\endcsname\relax\def\natexlab#1{#1}\fi
\providecommand{\url}[1]{\texttt{#1}}
\providecommand{\href}[2]{#2}
\providecommand{\path}[1]{#1}
\providecommand{\DOIprefix}{doi:}
\providecommand{\ArXivprefix}{arXiv:}
\providecommand{\URLprefix}{URL: }
\providecommand{\Pubmedprefix}{pmid:}
\providecommand{\doi}[1]{\href{http://dx.doi.org/#1}{\path{#1}}}
\providecommand{\Pubmed}[1]{\href{pmid:#1}{\path{#1}}}
\providecommand{\bibinfo}[2]{#2}
\ifx\xfnm\relax \def\xfnm[#1]{\unskip,\space#1}\fi
\bibitem[{Arratia et~al.(1989)Arratia, Goldstein, Gordon
  et~al.}]{arratia1989two}
\bibinfo{author}{Arratia, R.}, \bibinfo{author}{Goldstein, L.},
  \bibinfo{author}{Gordon, L.}, et~al., \bibinfo{year}{1989}.
\newblock \bibinfo{title}{Two moments suffice for poisson approximations: the
  chen-stein method}.
\newblock \bibinfo{journal}{The Annals of Probability} \bibinfo{volume}{17},
  \bibinfo{pages}{9--25}.
\bibitem[{Bentkus and Rudzkis(1983)}]{bentkus1983distribution}
\bibinfo{author}{Bentkus, R.Y.}, \bibinfo{author}{Rudzkis, R.},
  \bibinfo{year}{1983}.
\newblock \bibinfo{title}{On the distribution of some statistical estimates of
  spectral density}.
\newblock \bibinfo{journal}{Theory of Probability \& Its Applications}
  \bibinfo{volume}{27}, \bibinfo{pages}{795--814}.
\bibitem[{Brockwell and Davis(2006)}]{Brockwell2006}
\bibinfo{author}{Brockwell, P.J.}, \bibinfo{author}{Davis, R.A.},
  \bibinfo{year}{2006}.
\newblock \bibinfo{title}{Time series: theory and methods}.
\newblock Springer Series in Statistics, \bibinfo{publisher}{Springer, New
  York}.
\newblock \bibinfo{note}{Reprint of the second (1991) edition}.
\bibitem[{Cai et~al.(2013)Cai, Ma et~al.}]{cai2013optimal}
\bibinfo{author}{Cai, T.T.}, \bibinfo{author}{Ma, Z.}, et~al.,
  \bibinfo{year}{2013}.
\newblock \bibinfo{title}{Optimal hypothesis testing for high dimensional
  covariance matrices}.
\newblock \bibinfo{journal}{Bernoulli} \bibinfo{volume}{19},
  \bibinfo{pages}{2359--2388}.
\bibitem[{Dette and D{\"o}rnemann(2020)}]{dette2020likelihood}
\bibinfo{author}{Dette, H.}, \bibinfo{author}{D{\"o}rnemann, N.},
  \bibinfo{year}{2020}.
\newblock \bibinfo{title}{Likelihood ratio tests for many groups in high
  dimensions}.
\newblock \bibinfo{journal}{Journal of Multivariate Analysis} ,
  \bibinfo{pages}{104605}.
\bibitem[{Eichler(2008)}]{eichler2008testing}
\bibinfo{author}{Eichler, M.}, \bibinfo{year}{2008}.
\newblock \bibinfo{title}{Testing nonparametric and semiparametric hypotheses
  in vector stationary processes}.
\newblock \bibinfo{journal}{Journal of Multivariate Analysis}
  \bibinfo{volume}{99}, \bibinfo{pages}{968--1009}.
\bibitem[{Embrechts et~al.(2013)Embrechts, Kl{\"u}ppelberg and
  Mikosch}]{embrechts2013modelling}
\bibinfo{author}{Embrechts, P.}, \bibinfo{author}{Kl{\"u}ppelberg, C.},
  \bibinfo{author}{Mikosch, T.}, \bibinfo{year}{2013}.
\newblock \bibinfo{title}{Modelling extremal events: for insurance and
  finance}. volume~\bibinfo{volume}{33}.
\newblock \bibinfo{publisher}{Springer Science \& Business Media}.
\bibitem[{Fan et~al.(2019)Fan, Jiang et~al.}]{fan2019largest}
\bibinfo{author}{Fan, J.}, \bibinfo{author}{Jiang, T.}, et~al.,
  \bibinfo{year}{2019}.
\newblock \bibinfo{title}{Largest entries of sample correlation matrices from
  equi-correlated normal populations}.
\newblock \bibinfo{journal}{The Annals of Probability} \bibinfo{volume}{47},
  \bibinfo{pages}{3321--3374}.
\bibitem[{Jiang et~al.(2004)}]{jiang2004asymptotic}
\bibinfo{author}{Jiang, T.}, et~al., \bibinfo{year}{2004}.
\newblock \bibinfo{title}{The asymptotic distributions of the largest entries
  of sample correlation matrices}.
\newblock \bibinfo{journal}{The Annals of Applied Probability}
  \bibinfo{volume}{14}, \bibinfo{pages}{865--880}.
\bibitem[{Lin and Liu(2009)}]{lin2009maxima}
\bibinfo{author}{Lin, Z.}, \bibinfo{author}{Liu, W.}, \bibinfo{year}{2009}.
\newblock \bibinfo{title}{On maxima of periodograms of stationary processes}.
\newblock \bibinfo{journal}{The Annals of Statistics} ,
  \bibinfo{pages}{2676--2695}.
\bibitem[{Liu and Wu(2010)}]{liu2010asymptotics}
\bibinfo{author}{Liu, W.}, \bibinfo{author}{Wu, W.B.}, \bibinfo{year}{2010}.
\newblock \bibinfo{title}{Asymptotics of spectral density estimates}.
\newblock \bibinfo{journal}{Econometric Theory} , \bibinfo{pages}{1218--1245}.
\bibitem[{Loubaton and Mestre(2020)}]{loubaton2020asymptotic}
\bibinfo{author}{Loubaton, P.}, \bibinfo{author}{Mestre, X.},
  \bibinfo{year}{2020}.
\newblock \bibinfo{title}{On the asymptotic behaviour of the eigenvalue
  distribution of block correlation matrices of high-dimensional time series}.
\newblock \bibinfo{journal}{arXiv preprint arXiv:2004.07226} .
\bibitem[{Loubaton and Rosuel(2021)}]{loubaton2021large}
\bibinfo{author}{Loubaton, P.}, \bibinfo{author}{Rosuel, A.},
  \bibinfo{year}{2021}.
\newblock \bibinfo{title}{Large random matrix approach for testing independence
  of a large number of gaussian time series, v4}.
\newblock \href{http://arxiv.org/abs/2007.08806}{{\tt arXiv:2007.08806}}.
\bibitem[{Mestre and Vallet(2017)}]{mestre2017correlation}
\bibinfo{author}{Mestre, X.}, \bibinfo{author}{Vallet, P.},
  \bibinfo{year}{2017}.
\newblock \bibinfo{title}{Correlation tests and linear spectral statistics of
  the sample correlation matrix}.
\newblock \bibinfo{journal}{IEEE Transactions on Information Theory}
  \bibinfo{volume}{63}, \bibinfo{pages}{4585--4618}.
\bibitem[{Morales-Jimenez et~al.(2018)Morales-Jimenez, Johnstone, McKay and
  Yang}]{morales2018asymptotics}
\bibinfo{author}{Morales-Jimenez, D.}, \bibinfo{author}{Johnstone, I.M.},
  \bibinfo{author}{McKay, M.R.}, \bibinfo{author}{Yang, J.},
  \bibinfo{year}{2018}.
\newblock \bibinfo{title}{Asymptotics of eigenstructure of sample correlation
  matrices for high-dimensional spiked models}.
\newblock \bibinfo{journal}{arXiv preprint arXiv:1810.10214} .
\bibitem[{Pan et~al.(2014)Pan, Gao and Yang}]{pan2014testing}
\bibinfo{author}{Pan, G.}, \bibinfo{author}{Gao, J.}, \bibinfo{author}{Yang,
  Y.}, \bibinfo{year}{2014}.
\newblock \bibinfo{title}{Testing independence among a large number of
  high-dimensional random vectors}.
\newblock \bibinfo{journal}{Journal of the American Statistical Association}
  \bibinfo{volume}{109}, \bibinfo{pages}{600--612}.
\bibitem[{Resnick(2013)}]{resnick2013extreme}
\bibinfo{author}{Resnick, S.I.}, \bibinfo{year}{2013}.
\newblock \bibinfo{title}{Extreme values, regular variation and point
  processes}.
\newblock \bibinfo{publisher}{Springer}.
\bibitem[{Robert(2007)}]{Robert2007}
\bibinfo{author}{Robert, C.P.}, \bibinfo{year}{2007}.
\newblock \bibinfo{title}{The {B}ayesian choice}.
\newblock Springer Texts in Statistics. \bibinfo{edition}{second} ed.,
  \bibinfo{publisher}{Springer, New York}.
\newblock \bibinfo{note}{From decision-theoretic foundations to computational
  implementation}.
\bibitem[{Rosuel et~al.(2020)Rosuel, Vallet, Loubaton and
  Mestre}]{rosuel2020frequency}
\bibinfo{author}{Rosuel, A.}, \bibinfo{author}{Vallet, P.},
  \bibinfo{author}{Loubaton, P.}, \bibinfo{author}{Mestre, X.},
  \bibinfo{year}{2020}.
\newblock \bibinfo{title}{On the frequency domain detection of high dimensional
  time series}, in: \bibinfo{booktitle}{ICASSP 2020-2020 IEEE International
  Conference on Acoustics, Speech and Signal Processing (ICASSP)},
  \bibinfo{organization}{IEEE}. pp. \bibinfo{pages}{8782--8786}.
\bibitem[{Rudelson et~al.(2013)Rudelson, Vershynin et~al.}]{rudelson2013hanson}
\bibinfo{author}{Rudelson, M.}, \bibinfo{author}{Vershynin, R.}, et~al.,
  \bibinfo{year}{2013}.
\newblock \bibinfo{title}{Hanson-wright inequality and sub-gaussian
  concentration}.
\newblock \bibinfo{journal}{Electronic Communications in Probability}
  \bibinfo{volume}{18}.
\bibitem[{Rudin(1987)}]{rudin1987real}
\bibinfo{author}{Rudin, W.}, \bibinfo{year}{1987}.
\newblock \bibinfo{title}{Real and Complex Analysis}.
\newblock Higher Mathematics Series, \bibinfo{publisher}{McGraw-Hill
  Education}.
\bibitem[{Rudzkis(1985)}]{rudzkis1985distribution}
\bibinfo{author}{Rudzkis, R.}, \bibinfo{year}{1985}.
\newblock \bibinfo{title}{On the distribution of the maximum deviation of the
  gaussian stationary time series spectral density estimate}.
\newblock \bibinfo{journal}{Lithuanian Mathematical Journal}
  \bibinfo{volume}{25}, \bibinfo{pages}{18--130}.
\bibitem[{Shao et~al.(2007)Shao, Wu et~al.}]{shao2007asymptotic}
\bibinfo{author}{Shao, X.}, \bibinfo{author}{Wu, W.B.}, et~al.,
  \bibinfo{year}{2007}.
\newblock \bibinfo{title}{Asymptotic spectral theory for nonlinear time
  series}.
\newblock \bibinfo{journal}{The Annals of Statistics} \bibinfo{volume}{35},
  \bibinfo{pages}{1773--1801}.
\bibitem[{Wahba(1971)}]{wahba1971some}
\bibinfo{author}{Wahba, G.}, \bibinfo{year}{1971}.
\newblock \bibinfo{title}{Some tests of independence for stationary
  multivariate time series}.
\newblock \bibinfo{journal}{Journal of the Royal Statistical Society: Series B
  (Methodological)} \bibinfo{volume}{33}, \bibinfo{pages}{153--166}.
\bibitem[{Walker(1965)}]{walker1965some}
\bibinfo{author}{Walker, A.}, \bibinfo{year}{1965}.
\newblock \bibinfo{title}{Some asymptotic results for the periodogram of a
  stationary time series}.
\newblock \bibinfo{journal}{Journal of the Australian Mathematical Society}
  \bibinfo{volume}{5}, \bibinfo{pages}{107--128}.
\bibitem[{Woodroofe and Van~Ness(1967)}]{woodroofe1967maximum}
\bibinfo{author}{Woodroofe, M.B.}, \bibinfo{author}{Van~Ness, J.W.},
  \bibinfo{year}{1967}.
\newblock \bibinfo{title}{The maximum deviation of sample spectral densities}.
\newblock \bibinfo{journal}{The Annals of Mathematical Statistics} ,
  \bibinfo{pages}{1558--1569}.
\bibitem[{Wu and Zaffaroni(2018)}]{wu2018asymptotic}
\bibinfo{author}{Wu, W.B.}, \bibinfo{author}{Zaffaroni, P.},
  \bibinfo{year}{2018}.
\newblock \bibinfo{title}{Asymptotic theory for spectral density estimates of
  general multivariate time series}.
\newblock \bibinfo{journal}{Econometric Theory} \bibinfo{volume}{34},
  \bibinfo{pages}{1--22}.

\end{thebibliography}








\end{document}